\documentclass[a4paper, 3p, 11pt]{elsarticle}
\usepackage{mystyle,hyperref,cleveref}

\hypersetup{
  colorlinks   = true, 
  urlcolor     = red, 
  linkcolor    = red, 
  citecolor   = red 
}
\graphicspath{{./pic/}}

\newcommand{\dom}{\operatorname{dom}}
\biboptions{authoryear,compress}

\usepackage{amsthm}
\usepackage{thmtools}
\declaretheorem{theorem}
\declaretheorem[name=\textbf{Remark},style=definition]{Remark}
\declaretheorem[name=\textbf{Definition},style=definition]{definition}
\declaretheorem{proposition}
\declaretheorem{example}
\newcommand{\thref}[1]{\Cref{#1}}
\newcommand{\st}{\text{ s.t. }}
\newcommand{\cP}{\mathcal{P}}
\newcommand{\cU}{\mathcal{U}}
\newcommand{\varopt}{\mathbf{x}}
\newcommand{\val}{\operatorname{val}}

\renewcommand{\todo}[1]{}

\begin{document}

\begin{frontmatter}
\author[Navier]{Jeremy Bleyer\corref{cor1}}
\author[Cermics]{Vincent Leclère}

\address[Navier]{Laboratoire Navier, ENPC, Univ Gustave Eiffel, CNRS, Marne-la-vallée, France}

\address[Cermics]{CERMICS, ENPC, Marne-la-vallée, France}

 \cortext[cor1]{\textsl{Correspondence to}: J. Bleyer, Laboratoire Navier, 6-8 av Blaise Pascal, Cit\'e Descartes, 77455 Champs-sur-Marne, France, Tel : +33 (0)1 64 15 37 43, Email: \texttt{jeremy.bleyer@enpc.fr}}

\title{Robust limit analysis theory for computing worst-case limit loads under uncertainties}

\begin{abstract}
This work proposes a novel theoretical framework of robust limit analysis i.e. the computation of limit loads of structures  in presence of uncertainties using limit analysis and robust optimization theories. We first derive generic robust limit analysis formulations in the case of uncertain material strength properties. We discuss various ways of modeling uncertain strength properties and introduce the notion of robust strength criteria. 
We formulate static and adjustable robust counterparts of the corresponding uncertain limit analysis problems. Depending on the chosen strength uncertainty model, we also discuss tractable reformulations or approximations which can be implemented numerically using conic programming solvers. 
Building upon these results, we also derive robust limit analysis formulations in presence of loading uncertainties. 
Finally, various applications illustrate the versatility of the proposed framework including the derivation of a robust Mohr-Coulomb criterion with uncertain cohesion and friction angle, the computation of limit load of a structure subject to the partial loss of some components or the limit load of a truss structure under multiple load uncertainties.
\end{abstract}

\end{frontmatter}
\section{Introduction}
In the context of perfect plasticity, the theory of limit analysis \citep{hill1950mathematical} provides a powerful direct method for estimating the limit load of a structure, relying only on the compatibility between equilibrium conditions on the one hand and plasticity conditions on the other hand. 
It is based on two bounding theorems based either on a static formulation or a kinematic formulation. The static approach provides a lower bound estimate of the limit load when restricting to a discretized set of admissible stresses. 
Conversely, the dual kinematic approach relies on the search for critical failure mechanisms and provides an upper bound estimate of the limit load. In the yield design theory \citep{salencon1983calcul,salenccon2013yield}, plasticity yield conditions are replaced by a generic material strength criterion. One then obtain  an estimate of the structure ultimate load, provided that the material presents sufficient ductility. 
Applications of this generic framework cover a wide range of mechanical and civil engineering problems including geotechnical problems (slope stability, footing capacity, etc.) \citep{chen2013limit}, reinforced-concrete structures \citep{chen2007plasticity, nielsen2016limit} or slabs \citep{johansen1962yield}, limit load of steel frames, plates or shell structures \citep{save1995atlas,save1997plastic}, etc. 
In particular, this kind of reasoning forms the basis of Ultimate Limit State design in civil engineering design codes.

The main advantage of limit analysis is that it avoids the need of solving a step-by-step elastoplastic analysis until plastic failure, the latter being in general detected by observing large displacements or non-convergence of the resolution algorithm. 
Limit analysis problems can instead be formulated as convex optimization programs and its numerical implementation has recently benefited from the development of efficient convex optimization solvers.
 More precisely, most strength criteria can be expressed using second-order cone or semi-definite constraints \citep{bisbos2007second,makrodimopoulos2010remarks}, resulting in so-called second-order cone programs (SOCP) or semi-definite programs (SDP) respectively. 
 Based on the seminal work of \citet{karmarkar1984new} who developed  \textit{interior-point} (IP) algorithms for linear programming (LP), interior-point algorithms have later been extended to SOCP and SDP with great efficiency and robustness \citep{nesterov1994interior}. 
 Combined with a specific spatial discretization, e.g. the finite-element method, the use of IP solvers is today the state-of-the-art approach for solving large-scale limit analysis problems. The reader can refer to \citep{krabbenhoft2008three,martin_finite-element_2008,portioli2014limit,vincent2018yield} for some applications and to \citep{bleyer2019automating} for a recent overview on this matter.\\

In the corresponding limit analysis convex formulation, given data consist of a known material yield criterion, a known loading (variable and fixed parts) and a known geometry. However, in real applications, uncertainty may affect all these data, e.g. the amplitude or direction of some loading may not be known exactly, material strength properties can vary locally, geometry may not always be precisely known, etc. 
Furthermore, engineers are interested in designing a structure which would be robust to such uncertainties, meaning that the design load should always be safe for all possible realizations of the uncertain parameters. This requires designing the structure against a \textit{worst-case situation}. 
By varying the level of uncertainty, robustness of different constructive systems could therefore be assessed by comparing how much the collapse load depends on the uncertainty level. 
In doing so, one would distinguish fragile structures (i.e. sensitive to uncertainties) from robust structures. 
Unfortunately, it is very difficult to systematically anticipate the worst-case scenario or to explicitly explore all realizations of the uncertainty.\\

As regards practical applications, uncertainties are important to consider in the design stage in many cases. 
For instance, in geomechanics, soils are known to exhibit large spatial variability of their mechanical properties, especially in the depth direction. 
Available properties measurements are often very scarce in practice so that one has to resort to very conservative values of the soil properties. Moreover, for cohesive-frictional soils, shear strength and friction angle are in general negatively correlated \citep{wang2016quantifying}. 
It is therefore not necessarily always obvious what would be good worst-case values to retain in a robust design. 
Similarly, many heterogeneous materials are modeled using a homogeneous effective strength criterion \citep{suquet1985elements,de1986fundamental} assuming a well-defined microstructure, e.g. periodic masonry \citep{de1997homogenization}, jointed rocks \citep{taliercio1988failure}, fiber-reinforced materials, etc. 
Usually, real materials are not exactly periodic or fractures/fibers are not always perfectly aligned along a given direction. 
In these cases, the effective properties computed using homogenization theory may overestimate the real strength capacity of the material.

Accounting for loading uncertainty is also crucial for a robust design as many simplifying hypotheses are usually assumed when considering multiple loading parameters. For instance, assessing the bearing capacity of a footing usually assumes a single vertical loading. 
In order to account for additional lateral loading or applied bending moment to the footing, complex interaction diagrams must be determined for all possible loading combinations, inducing large computational costs \citep{kusakabe2010bearing,graine20213d}. 
To simplify this procedure, it would therefore be relevant to compute the reduced bearing capacity relative to a dominant vertical loading accounting for uncertainties in the value of secondary loadings (lateral force or bending moment). 

Another important application concerns the assessment of structural robustness which aims at evaluating the residual capacity of a structure under a catastrophic event (explosion, column removal, etc.). 
Such kind of events is usually not accounted for at the design stage and assessing structural robustness often relies on a scenario-based approach. Critical scenarios must be defined and individually computed through complex non-linear computations \citep{fu2009progressive, diab2021progressive}. 
This approach is therefore limited to the definition of appropriate scenarios which cannot be too numerous due to the high numerical cost induced when evaluating them.
 Combination of scenarios is also usually approximate due to the non-linearity of the problem.
 In this context, it would also be relevant to devise a procedure accounting for a potential failure of some elements without having to define and explore individually each scenario.
\\

Accounting for uncertainty in mechanical problems can be achieved using stochastic models.
 This approach requires the introduction of appropriate probability distributions of the uncertain parameters.
 As regards plastic limit analysis, general definitions of the probability of collapse have been given in \citep{salencon1983calcul,augusti1984probabilistic}, later revisited by \citet{alibrandi2008use} using stochastic stress vectors.
 Various works have also considered the numerical computation of limit loads in a stochastic setting such as \citep{yang2012system,staat2014limit} or \citep{bjerager1983influence,huang2013quantitative,kasama2016effect,ali2016application} for geotechnical applications.
 For instance, the reader can refer to \citep{jiang2022advances} for a recent review of slope stability in spatially variable soil

Alternative approaches seek to assess structural robustness or reliability through non-probabilistic approaches.
 For instance, \citet{ben1994non,ben1995non} relies on a convex model of uncertainty to define reliability of linear systems.
 In \citep{matsuda2008robustness}, the authors consider uncertain limit analysis of truss structures with very similar sources of uncertainties as those investigated in this work.
 For this purpose, they used the info-gap decision theory \citep{ben2006info} which is however known to be difficult to apply in practice since robustness functions are very hard to compute in general.
 For the very specific case of truss structures, \citet{matsuda2008robustness} however manage to compute it via the resolution of a LP problem.
 Similarly, mixed-integer programming approaches can also be used to compute a worst-case limit load \citep{kanno2007worst} but solving such NP-hard problems is notoriously difficult and almost impossible for large-scale problems.
 Using a chance-constrained programming approach, \citet{tran2018shakedown, tran2020direct} considered limit analysis and shakedown theorems under normal or log-normal strength uncertainties for von Mises plasticity.
  In this simple setting, chance-constrained programs can be explicitly reformulated as convex deterministic problems with reduced strength properties (this result will be recovered later in section \ref{sec:homothetic-uncertainty}).
  \citet{tran2021direct} also considered random loadings but the resulting deterministic problem turns out to be non-convex which is highly undesirable for large-scale numerical applications.
 \\

In this work, we follow the same spirit as the previous works handling uncertainty using non-probabilistic convex models but using a different route.
 We want to explore developments in robust optimization theory \citep{ben2009robust, bertsimas2011theory} to devise a robust counterpart of standard limit analysis problems in order to compute, in an affordable manner, a collapse load estimate corresponding to, or at least safely approximating, the worst-case limit load for all possible uncertainty realizations.
 Robust Optimization (RO) is a relatively new and active field of research with a growing number of applications  \citep{bertsimas2004price, bertsimas2011theory, gabrel2014recent} in finance, supply chain, energy, healthcare, etc.
 RO is an appealing way of handling uncertainty in optimization problems as it does not assume any probability distribution for the uncertain parameters but relies instead on the notion of \textit{uncertainty sets}.
 More precisely, it seeks to optimize some objective function provided that the optimization (or decision) variables remain feasible for any realization of the uncertain data in the uncertainty set.
 The popularity of RO can be attributed to the fact that dealing with such hard uncertain constraints can be reformulated as standard deterministic constraints for various types of problems and uncertainty sets.
 This notion of \textit{tractability} \citep{bertsimas2011theory} is one of the key aspects in deriving robust counterparts in the RO framework.
 A second important aspect is the \textit{conservatism} of RO formulations \citep{bertsimas2004price}.
 In its most basic formulation, RO considers only static decision variables, the value of which are not affected by the uncertainty realization.
 In some cases, this results in overly-conservative or even leads to infeasible solutions.
 Adjustable Robust Optimization \citep{yanikouglu2019survey} considers also decision variables to be adjustable to the uncertainty realization.
 The corresponding adjustable robust counterpart is then less conservative than its initial static counterpart but also much more complex to solve and involves solving procedure often relying on appropriate decision rules \citep{ben2009robust,gorissen2015practical}.\\

Benefiting from the deep conceptual link between limit analysis and convex optimization, we endeavour to formulate a robust limit analysis theory.
 Contrary to previously mentioned contributions using info-gap or chance-constrained programming, we aim here at providing a generic framework which can be applied to different mechanical models (solids, trusses, beams, shells, etc.) and generic strength criteria (polyhedral or non-polyhedral ones).
  Most importantly regarding computational efficiency, we concentrate on obtaining tractable robust counterparts i.e. convex programs the size of which scales at most polynomially with the deterministic problem data.
   This important constraint will however allow us to benefit from the efficiency of IP solvers as mentioned previously when considering large-scale realistic applications.\\

The manuscript is organized as follows: section \ref{sec:generic-formulations} introduces robust formulations of limit analysis theory in the case of strength uncertainty; section \ref{sec:RC-strength-uncertainty} 
details the derivation of tractable robust counterparts of uncertain strength constraints arising in the previous formulations;
 section \ref{sec:AARC-main} is devoted to the specific case of affinely adjustable robust counterparts;
  finally, section \ref{sec:load-uncertainty} deal with the case of uncertain loads and section \ref{sec:conclusions} draws some conclusions and perspectives for future research.

\todo{Check for min/max vs. inf/sup occurrences}
\section{Generic formulations of robust limit analysis theory}\label{sec:generic-formulations}

\subsection{A robust optimization primer}
\label{ssec:robust_optim_primer}
In this section we give a quick overview of the \emph{robust optimization} framework and tools we are using in this paper.

\subsubsection{Optimization under uncertainty}
\label{ssec:optim_under_uncertainty}

We consider the following \emph{nominal} optimization problem,
\begin{subequations}
  \label{pb:P_zeta}
  \begin{align}
   (\cP_{\bzeta_0}) \qquad  \val(\bzeta_0) \quad =\quad 
    \max_{\varopt} \quad &  J(\varopt) \\
    \st \quad & g(\varopt) \in G(\bzeta_0) 
    \label{cst:zeta}
  \end{align}
\end{subequations}
  that is we are looking for the vector $\varopt$, that maximizes the  objective function $J$,
  while satisfying the constraint $g(\varopt) \in G(\bzeta_0)$ for some set $G$ and constraint function $g$.

  In the above problem, $\bzeta_0$ is a vector of given (nominal) parameters. 
  Unfortunately, in numerous applications, the parameters $\bzeta_0$ are not well known, and the vast field of optimization under uncertainty tackles this fact through various approaches.
  For instance, \emph{chance-constrained programming} consists in assuming that $\bzeta$ is a random variable, and replaces the constraint~\eqref{cst:zeta} by 
  \begin{equation}
    \label{cst:chance_constraint}
    \mathbb{P} \left( g(\varopt) \in G(\bzeta_0) \right) \geq 1- \varepsilon
  \end{equation}  
  where $\varepsilon$ is a confidence parameter.
  While the chance constraint~\eqref{cst:chance_constraint} is intuitive enough, it is generally mathematically difficult and requires to know the exact law of the uncertain parameter, which is often hard to obtain.
  \emph{Robust Optimization} offers a third approach, by simply assuming that the uncertain parameter $\bzeta$ lies in an \emph{uncertainty set} $\cU$, and that the optimal solution $\varopt_R$ shall maximize the objective function while satisfying the constraint \emph{for all possible values of $\bzeta$}, that is~\eqref{cst:zeta} now reads
  \begin{equation}
    \label{cst:robust}
    g(\varopt) \in G(\bzeta) \quad \forall \bzeta \in \cU
    \qquad \text{or equivalently} \qquad
    g(\varopt) \in \bigcap_{\bzeta \in \cU} G(\bzeta)
  \end{equation}
  \begin{Remark}
    \label{rk:sensitivity_analysis}
  Both of these approaches are fundamentally different from \emph{sensitivity analysis}, which consists in solving the problem $(\cP_{\bzeta})$ for various values of $\bzeta$ to observe the variation of the value of the problem $\val(\bzeta)$ with the parameter $\bzeta$. 
  Indeed, in sensitivity analysis the optimal solution $\varopt_{\bzeta}$ is chosen while knowing the actual value of the uncertain parameter $\bzeta$, while the optimal solution to a chance constrained or robust problem is the same, and take into account, all possible values of $\bzeta$.
\end{Remark}

A difficulty of robust optimization modeling is the choice of the uncertainty set $\cU$. 
This choice is made with two different objectives in mind: representativity of uncertainty and \emph{tractability} of the induced problem. 
Representativity means that $\cU$ should encompass all or most of the reasonable possible values of $\bzeta$.\footnote{This is sometimes topped-up by theorems saying that if $\bsig$ satisfies~\eqref{cst:robust} for some uncertainty set $\cU$, then it satisfies the chance constraint~\eqref{cst:chance_constraint} for some $\varepsilon$ and a large class of law of $\bzeta$. 
Discussion of such \emph{probabilistic guarantees} falls out of the scope of this paper.
}
Tractability means that the robust optimization problem can be numerically (approximately) solved efficiently.
Examples of classical uncertainty sets are given in \cref{sec:uncertainty-sets}.

When the uncertainty set is not tractable, a common method consists in approximating $\cU$ by a more tractable uncertainty set $\hat{\cU}$. 
We say that $\hat{\cU}$ is a \emph{safe approximation} of $\cU$ if any $\varopt$ satisfying the robust constraint \eqref{cst:robust} with $\hat{\cU}$ also satisfies the original constraint \eqref{cst:robust} with $\cU$.
For example, this is the case if $\cU \subseteq \hat{\cU}$.

\subsubsection{Typical uncertainty sets}\label{sec:uncertainty-sets}
Robust optimization offers great freedom in modeling uncertainty through the choice of the uncertainty set.
 The latter should however remain simple enough to ensure tractability of the robust counterpart. 
 Typical examples of such simple uncertainty sets often encountered in the literature are :
\begin{itemize}
\item the \textit{box uncertainty} or $\ell_\infty$-norm:
\begin{equation}
\Uu_\infty = \{\bzeta \in \RR^m \st \|\bzeta\|_\infty = \max_j |\zeta_j|\leq 1\}
\end{equation}
\item the \textit{ellipsoid uncertainty} or $\ell_2$-norm:
\begin{equation}
\Uu_2 = \{\bzeta \in \RR^m \st \|\bzeta\|_2 = \sqrt{\sum_j \zeta_j^2}\leq 1\}
\end{equation}
\item the \textit{cross-polytope uncertainty} or $\ell_1$-norm:
\begin{equation}
\Uu_1 = \{\bzeta \in \RR^m \st \|\bzeta\|_1 = \sum_j |\zeta_j|\leq 1\}
\end{equation}
\item the \textit{budget uncertainty} or $\ell_1\cap \ell_\infty$-norm:
\begin{equation}
\Uu_\text{budget}(\Gamma) = \{\bzeta \in \RR^m \st \|\bzeta\|_\infty \leq 1 \text{ and } \|\bzeta\|_1 \leq \Gamma\}
\label{eq:budget-uncertainty}
\end{equation}
\end{itemize}
The above uncertainty sets correspond to unit balls of standard norms for non-dimensional uncertainty parameters $\bzeta$. They can obviously be generalized to non-unit balls with radius $R$ or to parameters $\bzeta$ with different physical dimensions (see section \ref{sec:physical-uncertainty}).
The box uncertainty is the most conservative one as uncertain parameters can take all values between $\pm 1$ without restriction. 
Using such a set ignores correlation between the parameters. 
As a result, the obtained robust solutions are, in general, over-conservative in practice. 
The ellipsoid uncertainty has the advantage of being less conservative but it results in robust counterparts with higher complexity. 
For this reason, many interesting results in the robust optimization literature can be obtained when restricting to a polyhedral uncertainty.
 This makes the cross-polytope uncertainty quite appealing.
 It assumes that one uncertain parameter can take its maximum value only if all remaining parameters are zero.
 A generalization of this observation can be obtained using the budget uncertainty $\Uu_\text{budget}(\Gamma)$ which states that at most $\Gamma$ parameters can take their maximum value simultaneously.
 Note that if $\Gamma=1$, one recovers the cross-polytope uncertainty whereas the box uncertainty is obtained for $\Gamma\geq m$.
\\

\subsubsection{Robust optimization with recourse}
\label{ssec:robust_with_recourse}

Up to now, we have mainly discussed a \emph{static} robust optimization problem where a decision variable $\varopt$ has to be resilient to all possible uncertainty.
However, in some cases the optimization problem has two types of variable, first stage variables $\varopt$ that are the same for all uncertainty, and \emph{recourse} variables that can be adapted to the actual value of the uncertainty.
More precisely, we consider a problem of the form 
\begin{equation}
  \label{pb:ARC}
  (\text{ARC}) \qquad \max_{\varopt}\quad 
  \bigg\{ J(\varopt) \quad \bigg| \quad 
  \forall \bzeta \in \cU, \quad 
  \exists \bsig(\bzeta), \quad
  g(\varopt, \bsig(\bzeta)) \in G(\bzeta)
  \bigg\} .
\end{equation}
Those \emph{adaptable recourse} problems, are numerically challenging, essentially because the recourse variable $\bsig$ is a function of the uncertainty $\bzeta$.
A classical approach consists in restricting the class of function in which we can take $\bsig$, to obtain an approximated problem. 
Let's highlight this idea with two main function classes : constant and affine functions.
Assuming $\bsig$ to be constant consists in considering that $\bsig$ is a first-stage variable which does not adapt to the uncertainty, contrary to a recourse variable.
This leads to solving a standard static robust optimization problem, namely
\begin{equation}
  \label{pb:RC}
  (\text{RC}) \qquad \max_{\varopt, \bsig}\quad 
  \bigg\{ J(\varopt) \quad \bigg| \quad 
  \forall \bzeta \in \cU, \quad
  g(\varopt, \bsig) \in G(\bzeta)
  \bigg\} .
\end{equation}
Choosing $\bsig$ in the larger class of affine function consists in looking for a matrix $A$ and a vector $\bsig_0$ such that $\bsig(\bzeta)=A\bzeta + \bsig_0$ for all $\bzeta \in \cU$, yielding the following equivalent static optimization problem 
\begin{equation}
  \label{pb:AARC}
  (\text{AARC}) \qquad \max_{\varopt, A, \bsig_0}\quad 
  \bigg\{ J(\varopt) \quad \bigg| \quad 
  \forall \bzeta \in \cU, \quad
  g(\varopt, A\bzeta + \bsig_0) \in G(\bzeta)
  \bigg\} .
\end{equation}
Note that the problem (RC) is a more constrained version of problem (AARC) (as a constant function is a specific affine function) which his itself a more constrained version of problem (ARC),
thus we have the following inequality
\begin{equation}
  \val(\text{RC}) \leq \val(\text{AARC}) \leq \val(\text{ARC})
\end{equation}

\subsection{Nominal limit analysis}
Let us know apply the above-mentioned concepts to limit analysis problems. Formulated in its static form, a limit analysis problem amounts to finding the maximal value of the load factor such that one can find a stress field in equilibrium with the loading and satisfying the material local strength criterion.
 In the case of certain data, we will refer to it as the \textit{nominal} limit analysis problem and denote the corresponding maximum load factor as
$\lambda_\text{N}$.
 The latter can be given as the solution to the following maximization problem:
\begin{equation}
\begin{array}{rll}
\displaystyle{\lambda_\text{N}=\max_{\lambda, \bsig}} & \lambda &\\
\text{s.t.} & \div\bsig + \lambda \bf^\text{r} + \bf^\text{f} = 0  & \text{in }\Omega \\
& \bsig\cdot\bn = \lambda \bt^\text{r} + \bt^\text{f}  &\text{on } \partial \Omega_T \\
& \bsig \in G & \text{in }\Omega
\end{array} \tag{\texttt{N}}\label{nominal-LA}
\end{equation} 
where $\lambda$ is the load factor, $\bsig$ the Cauchy stress field in $\Omega$, $\bf^\text{r}$ (resp. $\bf^\text{f}$) the variable reference (resp. the fixed) body force, $\bt^\text{r}$ the variable 
reference (resp. the fixed) contact force prescribed on some part $\partial \Omega_T$ of the boundary and $G$ is the material yield/strength criterion which we assume to be a convex set (possibly unbounded) containing 0.

In the above, the first two constraints correspond to the local balance equation and traction boundary conditions, whereas the last one corresponds to the strength condition which must be satisfied at all point $\bx\in\Omega$. 
Note that the strength criterion $G$ could well be a function of $\bx$. 
For the sake of notational simplicity we do not make this dependence on $\bx$ explicit (similarly for $\bsig$, $\bf$, etc.). 
Similarly, we also drop the fact that the constraints must be enforced in $\Omega$ and on $\partial \Omega_T$ in the following.

\begin{Remark}
Note that this work make use of a 3D continuum formulation for the presentation of robust limit analysis problems. 
However, all the introduced notions can also equally apply to other mechanical models (e.g. truss, beam, plate or shell theories) by replacing the first two constraints with the corresponding local balance equations and stress boundary conditions and the strength criterion $G$ by the corresponding strength condition expressed in terms of generalized forces (e.g. normal force, bending moment, etc.).
\end{Remark}

\subsection{Uncertain limit analysis with strength uncertainties}
We first consider the simple case in which the loading is certain but the material may possess uncertain properties such that the strength criterion is now written as $G(\bzeta)$ where $\bzeta\in \Uu \subseteq \mathbb{R}^m$ is a vector of uncertain parameters and $\Uu$ the corresponding \textit{uncertainty set} in which the uncertainty must vary (see section \ref{sec:uncertainty-sets}). From now on, we assume $\Uu$ to be closed convex and full dimensional. 
 The maximum load factor now becomes uncertain i.e. it depends on the value of the uncertainty realization:
\begin{equation}
\begin{array}{rll}
\displaystyle{\lambda^+(\bzeta) =\max_{\lambda, \bsig}} & \lambda \\
\text{s.t.} & \div\bsig + \lambda \bf^\text{r} + \bf^\text{f} = 0 \\
& \bsig\cdot\bn = \lambda \bt^\text{r} + \bt^\text{f}   \\
& \bsig \in G(\bzeta) 
\end{array} \label{eq:uncertain-LA}
\end{equation}

The main purpose of robust optimization is to provide worst-case solutions to a given optimization problem.
 Our proposed theory of robust limit analysis  therefore aims at evaluating the worst-case limit load among all possible realizations.
In the remaining of this section, we discuss various robust formulations.

 \subsubsection{Adjustable robust optimization}
For a given loading and two different given realizations of the uncertainty, one expects that the stress field will be different depending on the uncertainty realizations. The most natural approach therefore consists in considering the stress field and the corresponding load factor to be recourse variables. Thus, we are faced with an \textit{adjustable robust counterpart} (ARC) to problem \eqref{eq:uncertain-LA} defined as follows:
 \begin{equation}
 \begin{array}{rll}
 \displaystyle{\lambda_\text{ARC} =\min_{\bzeta\in\Uu}\lambda^+(\bzeta) =\min_{\bzeta\in\Uu}\max_{\bsig(\bzeta),\lambda(\bzeta)} } & \lambda(\bzeta) \\
 \text{s.t.} & \div\bsig(\bzeta) + \lambda(\bzeta) \bf^\text{r} + \bf^\text{f} = 0 \\
 & \bsig(\bzeta)\cdot\bn = \lambda(\bzeta) \bt^\text{r} + \bt^\text{f}  \\
 & \bsig(\bzeta) \in G(\bzeta)  
 \end{array} \tag{\texttt{ARC}} \label{eq:adjustable-LA}
 \end{equation}
 i.e. we find the largest load factor such that, for each uncertainty realization there exists an optimal stress field in equilibrium, with the corresponding collapse load factor, satisfying the strength criterion.
 
 \todo{For polyhedral uncertainty $\Uu=\operatorname{conv} \{\bzeta_1,\ldots,\bzeta^K\}$, under which conditions do we have:
 $$\lambda_\text{ARC} =\min_{\bzeta\in\Uu}\lambda^+(\bzeta) = \min_{\bzeta\in\{\bzeta_1,\ldots,\bzeta^K\}}\lambda^+(\bzeta) $$
 ?
 }
 
 In the following, we also make use of the following equivalent formulation of the ARC problem \citep{takeda2008adjustable, marandi2018static}:
 \begin{equation}
 \begin{array}{rll}
 \displaystyle{\lambda_\text{ARC} =\max_{\bar\lambda}} & \bar\lambda & \\
 \text{s.t.} & \forall \bzeta \in \Uu, \exists \bsig,\lambda \st& \div\bsig + \lambda \bf^\text{r} + \bf^\text{f} = 0 \\
 & & \bsig\cdot\bn = \lambda \bt^\text{r} + \bt^\text{f}   \\
 & & \bsig \in G(\bzeta) \\
 & & \bar\lambda \leq \lambda
 \end{array} \label{eq:adjustable-LA-1}
 \end{equation}
 where uncertainty of the objective function has been transferred to the constraints with the introduction of a static (non-adjustable) variable $\bar\lambda$.
 

\subsubsection{Static robust optimization}\label{sec:static-RC-formulation}
Unfortunately, as noted in~\cref{ssec:robust_with_recourse}, adaptive recourse problem are numerically challenging. 
We follow here a conservative static robust counterpart (RC) in which we look for a stress field $\bsig$ and a load factor $\lambda$, independent of the exact realization of the uncertainty, which satisfy the strength condition $G(\bzeta)$ for all $\bzeta\in \Uu$.
 The corresponding problem can be formulated as follows:
\begin{equation}
\begin{array}{rll}
\displaystyle{\lambda_\text{RC} =\max_{\lambda, \bsig}} & \lambda &\\
\text{s.t.} & \div\bsig + \lambda \bf^\text{r} + \bf^\text{f} = 0 &\\
& \bsig\cdot\bn = \lambda \bt^\text{r} + \bt^\text{f}  & \\
& \bsig \in G(\bzeta) & \forall \bzeta \in \Uu
\end{array} \label{eq:static-RO-LA-1}
\end{equation}

What makes problem \eqref{eq:static-RO-LA-1} a \textit{robust optimization} problem is the condition $\forall\bzeta \in \Uu$ in the last constraint.
 This implies that the constraint $\bsig \in G(\bzeta)$ must be fulfilled for any possible value of $\bzeta \in \Uu$.
  It is therefore an infinite-dimensional constraint.
   One of the main goals of robust optimization theory is to make such a problem tractable using standard convex optimization algorithms.

For instance, the robust constraint can be reformulated as:
\begin{equation}
\bsig \in G(\bzeta) \quad \forall \bzeta \in \Uu  \quad \Leftrightarrow \quad \bsig \in G_\text{RC} \label{eq:robust-strength-condition}
\end{equation}
when introducing:
\begin{equation}
G_\text{RC} = \bigcap_{\bzeta\in\Uu} G(\bzeta) \label{eq:G-robust}
\end{equation} 
the robust counterpart to the uncertain strength criterion. 
In order for a stress field to be admissible with respect to any possible realization of the uncertain strength criterion $G(\bzeta)$, it has to belong to the intersection of all such domains (see Figure \ref{fig:robust-criterion}).

 Now, problem \eqref{eq:static-RO-LA-1} writes as:

\begin{equation}
\begin{array}{rll}
\displaystyle{\lambda_\text{RC} =\max_{\lambda, \bsig}} & \lambda &\\
\text{s.t.} & \div\bsig + \lambda \bf^\text{r} + \bf^\text{f} = 0 &\\
& \bsig\cdot\bn = \lambda \bt^\text{r} + \bt^\text{f}  & \\
& \bsig\in G_\text{RC}
\end{array} \tag{\texttt{RC}}\label{eq:static-RO-LA}
\end{equation}
which is now independent of the uncertainty realization. 
As a result, problem \eqref{eq:static-RO-LA} is a classical limit analysis problem with a different strength criterion given by \eqref{eq:G-robust}. 
This makes problem \eqref{eq:static-RO-LA} very appealing provided that a simple expression for $G_\text{RC}$ can be found. 
It is however very hard to determine a simple expression for the infinite-dimensional set intersection appearing in \eqref{eq:G-robust}. 
Exact or approximate reformulation of strength criteria robust counterparts  are discussed in section \ref{sec:RC-strength-uncertainty}.

\begin{figure}
\begin{center}
\includegraphics[width=0.5\textwidth]{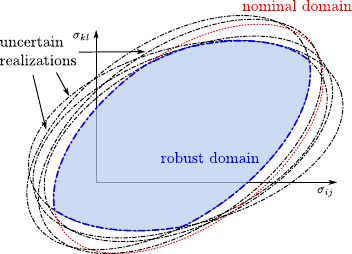}
\end{center}
\caption{Robust strength domain $G_\text{RC}$ (in blue) obtained as the intersection of various uncertain realizations $G(\bzeta)$ (in black) of a nominal domain (in red).} \label{fig:robust-criterion}
\end{figure}

\subsubsection{Affinely adjustable robust optimization}\label{sec:AARC-formulation}

Unfortunately, if (RC) problem are numerically tractable, the obtained approximation might be unreasonably conservative. 
As mentioned earlier, a middle ground is the \emph{affinely adjustable robust counterpart} (AARC), which consists in looking for 
adjustable variables $\bsig(\bzeta)$ and $\lambda(\bzeta)$ that are affine functions of the uncertain variable, the so-called \textit{affine decision rule} \citep{ben2004adjustable}:
\begin{subequations}
\label{eq:affine-decision-rule}
\begin{align}
\bsig(\bzeta) &= \bsig_0 + \sum_{j=1}^m \bsig_j\zeta_j \label{eq:affine-decision-rule-sig} \\
\lambda(\bzeta) &= \lambda_0 + \sum_{j=1}^m \lambda_j\zeta_j \label{eq:affine-decision-rule-lamb}
\end{align}
\end{subequations}
where the $\bsig_i$ (resp. $\lambda_i$) represent $1+m$ different stress fields (load factor variables) which are now static optimization variables. 
Inserting the affine decision rules \eqref{eq:affine-decision-rule-sig}-\eqref{eq:affine-decision-rule-lamb} into \eqref{eq:adjustable-LA}, the corresponding AARC reads:
\begin{equation}
\begin{array}{rll}
\displaystyle{\lambda_\text{AARC} =\max_{\bsig_i,\lambda_i}}\min_{\bzeta\in\Uu} & \displaystyle{\lambda_0 + \sum_{j=1}^m \lambda_j\zeta_j}\\
\text{s.t.} &\displaystyle{ \div\left(\bsig_0 + \sum_{j=1}^m \bsig_j\zeta_j\right)+ \left(\lambda_0 + \sum_{j=1}^m \lambda_j\zeta_j\right) \bf^\text{r} + \bf^\text{f} = 0 } \\
& \displaystyle{\left(\bsig_0 + \sum_{j=1}^m \bsig_j\zeta_j\right)\cdot\bn = \left(\lambda_0 + \sum_{j=1}^m \lambda_j\zeta_j\right)\bt^\text{r} + \bt^\text{f}  }\\
& \left(\displaystyle{\bsig_0 + \sum_{j=1}^m \bsig_j\zeta_j}\right) \in G(\bzeta) 
\end{array} \label{eq:AARC-LA}
\end{equation}
which can also be reformulated as follows:
\begin{equation}
\begin{array}{rlll}
\displaystyle{\lambda_\text{AARC} =\max_{\bar\lambda,\bsig_i,\lambda_i}} &\bar\lambda&\\
\text{s.t.} &\displaystyle{ \div(\bsig_0) + \lambda_0 \bf^\text{r} + \bf^\text{f} = 0 }& \\
& \displaystyle{ \div(\bsig_j) + \lambda_j \bf^\text{r}  = 0 } & \forall j=1,\ldots,m \\
& \displaystyle{\bsig_0\cdot\bn = \lambda_0 \bt^\text{r} + \bt^\text{f}  } & \\
& \displaystyle{\bsig_j\cdot\bn = \lambda_j\bt^\text{r}} & \forall j=1,\ldots,m\\
& \left(\displaystyle{\bsig_0 + \sum_{j=1}^m \bsig_j\zeta_j}\right) \in G(\bzeta) & \forall \bzeta\in \Uu\\
& \displaystyle{\bar\lambda \leq \lambda_0 + \sum_{j=1}^m \lambda_j\zeta_j} & \forall \bzeta\in \Uu
\end{array} \tag{\texttt{AARC}} \label{eq:AARC-LA-2}
\end{equation}
in which we removed the uncertainty from the objective function and replaced the minimization over $\bzeta$ with robust constraints. 
Note that equality constraints depending on $\bzeta$ have been re-expressed by identifying the corresponding terms of the expansion in terms of $\zeta_i$ since $\Uu$ is full dimensional.

\subsubsection{Comparison between the different approaches}
Summarizing, \eqref{eq:static-RO-LA} is the most conservative formulation yielding the smallest limit load. 
\eqref{eq:AARC-LA-2} is more flexible since it considers additional static variables $\bsig_j,\lambda_j$ for $j=1,\ldots,m$ and reduces to \eqref{eq:static-RO-LA} if we fix all $\bsig_j=0$. 
As mentioned, \eqref{eq:adjustable-LA} is less conservative than \eqref{eq:AARC-LA-2} since we allow for more general decision rules but is generally untractable. 
Finally, all of these formulations guard against all possible realizations of the uncertainty such that we have the following ordering:
\begin{equation}
\lambda_\text{RC} \leq \lambda_\text{AARC} \leq \lambda_\text{ARC} \leq \lambda^+(\bzeta) \quad \forall\bzeta\in\Uu
\end{equation}

\section{Tractable robust counterparts for strength uncertainties}\label{sec:RC-strength-uncertainty}

One major challenge of robust optimization is to obtain a tractable formulation of the robust problem. 
For instance, problem \eqref{eq:adjustable-LA} is almost never computationally tractable whereas \eqref{eq:static-RO-LA} may be tractable if $G_\text{RC}$ can be expressed in a simple fashion. 
Tractability strongly depends on the "shape" of the uncertainty set $\Uu$ and on how the uncertain constraints depend on $\bzeta$. 
For instance, robust linear programs are LP if the uncertainty set is polyhedral and SOCP if the uncertainty set is ellipsoidal. 
A robust second-order cone program  is SDP if the uncertainty is ellipsoidal. 
In complex cases where no tractable reformulations are known, one can devise tractable \textit{approximations} (which might be tight or not) to obtain an estimate, sometimes a safe one, of the robust solution.\\

In this section, we therefore discuss various tractable formulations of the static robust counterparts derived in section \ref{sec:static-RC-formulation} depending on different hypotheses made on the form of the uncertain strength properties, more precisely on how $G(\bzeta)$ may depend on $\bzeta$ and on the choice of the uncertainty set $\Uu$.

\subsection{Modeling uncertainties with uncertainty sets}

\subsubsection{Uncertainty on physical parameters}\label{sec:physical-uncertainty}
\todo{change ref}
The previous uncertainty sets have been defined for non-dimensional uncertain parameters $\bzeta$ with maximum absolute value of 1.
 Modeling uncertainty of physical parameters $\bk \in \RR^p$ using such sets $\Uu$ can be achieved through simple affine mappings of the form:
\begin{equation}
\bk(\bzeta) = \bk_0 + \bK\bzeta \quad \text{for } \bzeta \in \Uu \label{eq:physical-uncertainty}
\end{equation} 
where $\bk_0$ is the parameters nominal value and $\bK\in \RR^{p\times m}$ a matrix characterizing the physical amplitude of the parameters with respect to $\bzeta$.
 For instance, if $p=m$, $\bK = \diag(\Delta \bk)$ and $\Uu=\Uu_\infty$, uncertain parameters will vary in the following box region $\bk(\bzeta)\in [\bk_0-\Delta \bk; \bk_0+\Delta \bk]$.\\

In some cases, information on the cross-correlation between uncertain parameters might be available through a covariance matrix $\mathbf{Cov}$.
 Letting $\bK$ be the Cholesky factor of the covariance matrix ($\mathbf{Cov}=\bK\T\bK$), the use of $\Uu_2$ results in an ellipsoid domain with principal axis oriented along the direction of principal correlations.
  Similarly, the use of $\Uu_\infty$ (resp. $\Uu_1$) results in a skewed polytope domain circumscribed (resp. inscribed) to the corresponding ellipsoid.\\

\begin{Remark}
Note that robust optimization theory models uncertain parameters with bounded supports.
 This is a strong modeling choice which seems to exclude stochastic parameters following unbounded probability distributions (e.g. normal or log-normal).
  However, it is possible to devise probabilistic guarantees for which the satisfaction of a robust constraint ensures satisfaction of the stochastic constraint with sufficiently low probability.
   Typically, the "radius" of the uncertainty set can be scaled to satisfy the stochastic constraint with a sufficient confidence level.
    This aspect is however outside the scope of the present contribution.
\end{Remark}

\subsection{Homothetic strength uncertainty}\label{sec:homothetic-uncertainty}
Let us first suppose the case of a so-called \textit{homothetic strength uncertainty}, that is, the case where the strength uncertainty can be described by an uncertain scaling factor $\beta(\bzeta)$ while the strength criterion retains its nominal shape $G_0=G(\bzeta=0)$:
\begin{equation}
G(\bzeta) = (1-\beta(\bzeta))G_0 \label{eq:homothetic-uncertainty}
\end{equation}
where we assume that $\beta(\bzeta)=0$ for $\bzeta=0$.
 Note that, as mentioned earlier, since $G$ may also potentially depend on the space location, so does $\beta$.

In the following, it will be easier to describe the strength condition $\bsig\in G$ via the \textit{gauge function} (or \textit{Minkowski functional}) $g_G(\bsig)$ of the convex set $G$ defined as follows:
\begin{equation}
g_G(\bsig) = \inf \:\{\alpha\st \alpha > 0 \text{ and } \alpha\bsig \in G\}
\end{equation}
so that:
\begin{equation}
\bsig \in G \: \Longleftrightarrow \: g_G(\bsig)\leq 1
\end{equation}

For instance, for the homothetic strength uncertainty \eqref{eq:homothetic-uncertainty}, we have:
\begin{equation}
g_{G(\bzeta)}(\bsig) = \dfrac{g_{G_0}(\bsig)}{1-\beta(\bzeta)}
\end{equation}
The robust strength condition \eqref{eq:robust-strength-condition} can therefore be equivalently rewritten as:
\begin{align}
g_{G_0}(\bsig) &\leq 1-\beta(\bzeta) \quad \forall \bzeta\in\Uu \notag\\
g_{G_0}(\bsig) &\leq \min_{\bzeta\in \Uu} \{1-\beta(\bzeta)\} & \\
g_{G_0}(\bsig) &\leq 1-\max_{\bzeta\in \Uu} \beta(\bzeta) & \notag
\end{align}
Introducing $\bar{\beta}:= \max_{\bzeta\in \Uu} \beta(\bzeta)$, we see that $G_\text{RC}$ is then:
\begin{equation}
G_\text{RC} = (1-\bar{\beta})G_0 \label{eq:homothetic-RC}
\end{equation}
We can see that, even for this very simple case, obtaining the robust counterpart involves maximizing function $\beta$ over a convex set $\Uu$ which might not always be easy to do if $\beta$ and/or $\Uu$ have a complex structure.\\

In the case where $\beta(\bzeta)=\bb\T\bzeta$ is a linear function, we have 
\begin{equation}
\bar{\beta}=\max_{\bzeta\in \Uu} \bb\T\bzeta = \pi_\Uu(\bb) = \|\bb\|_*
\end{equation} 
where $\pi_\Uu$ is the support function of the uncertainty set $\Uu$ which is given by the corresponding dual norm $\|\cdot\|_*$ when uncertainty sets are unit balls of a given norm $\|\cdot\|$ (see \ref{app:dual-norms} for more details).

Finally, if $\beta(\bzeta) = f(\bb\T\bzeta)$ where $f$ is an increasing function, then we have:
\begin{equation}
\bar{\beta}=\max_{\bzeta\in \Uu} f(\bb\T\bzeta) = f(\|\bb\|_*)
\end{equation} 
Using these results, we obtain formulations quite similar to those from \citet{tran2018shakedown,tran2020direct,tran2021direct}.

\begin{Remark}
Note that the results obtained in this section can be generalized to the case where $G(\bzeta)$ is described as follows:
\begin{equation}
\bsig \in G(\bzeta) \: \Longleftrightarrow \: g_i(\bsig)\leq 1-b_i(\bzeta) \quad \forall i=1,\ldots, p
\end{equation} 
i.e. as the intersection of $p$ convex sets, each of them being associated with a gauge function $g_i$ and subjected to a homothetic uncertainty through a function $b_i(\bzeta)$. 
Note that, if the $b_i(\bzeta)$ are different, the uncertainty does not act in an homothetic manner on the whole intersection $G$.\\
The previous results can then be applied to each individual constraint independently of each other.
\end{Remark}

\subsection{Eroded strength conditions}
We now turn to a more generic type of uncertain strength conditions than the previous homothetic case, namely uncertain strength conditions of the following type:
\begin{equation}
g_G(\bsig + \bsig_\zeta) \leq 1,\quad \forall \bsig_\zeta \in \Uu_\Sigma
\label{eq:uncertain-effective-condition}
\end{equation}
where $\Uu_\Sigma$ is again a convex uncertainty set (here in the stress space).
 The robust counterpart of \eqref{eq:uncertain-effective-condition} defines a modified\footnote{Usually \textit{smaller} if $0\in\Uu_\Sigma$} strength domain which can be defined as the \textit{erosion} of $G$ by $\Uu_\Sigma$ using the erosion operator $\ominus$ appearing in morphological image processing \cite{serra1982image}.

\begin{definition}
The \textit{eroded strength domain} of $G$ with respect to the uncertainty set $\Uu_\Sigma$ is defined as:
\begin{equation}
G\ominus\Uu_\Sigma = \bigcap_{\bsig_\zeta \in \Uu_{\Sigma}} \{\hat{\bsig} - \bsig_\zeta \st \hat\bsig\in G\}
\label{eq:robust-effective-domain}
\end{equation}
\end{definition}
We have the following equivalent characterizations of the eroded strength domain:
\begin{proposition}
The following statements are equivalent:
\begin{itemize}
\item $\bsig \in G\ominus \Uu_\Sigma$
\item $g_G(\bsig + \bsig_\zeta) \leq 1,\quad \forall \bsig_\zeta \in \Uu_\Sigma$
\item $g_{G\ominus\Uu_\Sigma}(\bsig) = \sup_{\bsig_\zeta \in \Uu_\Sigma} g_G(\bsig+\bsig_\zeta)$
\end{itemize}
\end{proposition}

\begin{figure}
\begin{center}
\includegraphics[width=0.5\textwidth]{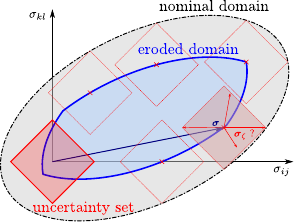}
\end{center}
\caption{Eroded strength domain $G\ominus\Uu$ (in blue) of the nominal domain $G$ (in black) with respect to the uncertainty set $\Uu_\Sigma$ (in red).}
\end{figure}

\begin{Remark}
Such eroded strength conditions can be related to yield criteria expressed in terms of effective stresses for porous materials \citep{coussy2004poromechanics}.
 For instance, in the case of Terzaghi's effective stress $\bsig_\text{eff} = \bsig - p\bI$, yield conditions are expressed on the effective stress: $\bsig_\text{eff}\in G$.
 If the pressure $p$ is known without any uncertainty, the corresponding strength domain in terms of the total stress $\bsig$ is given by the translation of $G$ by $p$ along the hydrostatic axis, that is the eroded domain $G\ominus\{p\bI\}$.
 However, if $p$ is known only within a given interval, say $p\in[p_\text{min};p_\text{max}]$, then the eroded domain is given by $G\ominus\{p\bI \st p\in[p_\text{min};p_\text{max}]\}$, that is as the intersection of $G$ translated by $p_\text{min}$ and $G$ translated by $p_\text{max}$ along the hydrostatic axis.
\end{Remark}

\subsection{Generic uncertain strength conditions and tractable approximations}\label{sec:generic-robust-strength-constraint}
Combining uncertain constraints of the form \eqref{eq:uncertain-effective-condition} and homothetic uncertainty with a linear function $\beta(\bzeta)=\bb\T\bzeta$, we now consider generic uncertain strength conditions of the form:
\begin{equation}
g_G(\bsig + \bsig_\zeta) \leq 1-\bb\T\bzeta,\quad \forall \bsig_\zeta \in \Uu_\Sigma, \forall\bzeta\in\Uu
\end{equation}
which reduces to a homothetic strength uncertainty if $\Uu_\Sigma=\{0\}$.
 Such kind of constraint appear frequently in the subsequent sections and therefore deserves some attention.
 To simplify notations and without loss of generality, we merge both types of uncertainties into a single uncertainty set $\Uu$ and consider instead the following type of uncertain conditions:
\begin{equation}
g(\bsig + \bSig\bzeta) \leq 1-\bb\T\bzeta, \quad \forall \bzeta \in \Uu \label{eq:generic-robust-strength-constraint}
\end{equation}
with $\bsig\in \RR^d, \bSig \in \RR^{d\times m}$, $d$ being the dimension of the stress space.\\

Unfortunately, an exact tractable reformulation of this constraint is not available in the general case since it involves maximizing the convex function $\bzeta \mapsto g(\bsig + \bSig\bzeta)$ over $\Uu$ which is very hard in practice.
 We now discuss cases where such an exact reformulation exists and provide approximations for the general case.
 
\subsubsection{Exact reformulation when $G$ is polyhedral}
In the specific case where $G$ is polyhedral, \eqref{eq:generic-robust-strength-constraint} admits an exact tractable reformulation. 
Indeed, let us consider $G$ to be a polyhedron described by the following set of $K$ linear inequalities:
\begin{equation}
G = \{\bsig \in \RR^d \st \bc_k\T\bsig \leq d_k, \: \forall k=1,\ldots,K\} 
\end{equation}

In this case, the robust strength condition \eqref{eq:generic-robust-strength-constraint} reads:
\begin{equation}
\bc_k\T(\bsig + \bSig\bzeta) \leq (1-\bb\T\bzeta)d_k \quad \forall \bzeta \in \Uu, \forall k=1,\ldots,K
\end{equation}
Since uncertainty is constraint-wise in robust optimization, we can consider only a single constraint of index $k$ and derive its corresponding worst-case formulation:
\begin{align}
& \bc_k\T\bsig + (\bc_k\T\bSig + d_k\bb\T)\bzeta \leq d_k \quad \forall \bzeta \in \Uu \\
\Leftrightarrow\quad & \bc_k\T\bsig + \max_{\bzeta \in \Uu} \{(\bc_k\T\bSig + d_k\bb\T)\bzeta\} \leq d_k
\end{align}
which, if $\cU$ is a centered ball for norm $\|\cdot\|$, reads $\bc_k\T\bsig + \|\bSig\T\bc_k + d_k\bb\|_* \leq d_k$ which is again tractable.

\begin{Remark}
Although the hypothesis of a polyhedral strength condition seems restrictive, there is still a large range of applications for limit analysis problems involving polyhedral strength conditions, including rigid block mechanics with Coulomb friction in 2D \citep{gilbert1994rigid}, some homogenized masonry strength criteria \citep{de1997homogenization}, yield line method for plates \citep{johansen1962yield}, linearizations of non-linear criteria \citep{pastor2004ductile}, etc.
\end{Remark}

\subsubsection{Exact reformulation when $\Uu$ is polyhedral}
As mentioned before, maximizing $g(\bsig + \bSig\bzeta)$ over $\Uu$ is very hard in the general case.
 However, when $\Uu$ is polyhedral, and $g$ convex, we know that the maximum is attained on one of the vertices of $\Uu$. As a result, we have the corresponding vertex-based reformulation:
\begin{theorem}[Vertex-based reformulation with polyhedral uncertainty]\label{th:vertex-reformulation}
If $\Uu$ is a polyhedron, i.e. the convex hull of vertices $\bzeta^k$, $k=1,\ldots,N$, then \eqref{eq:generic-robust-strength-constraint} is equivalent to:
\begin{equation}
\max_{k=1,\ldots,N} \:\{g(\bsig + \bSig\bzeta^k) + \bb\T\bzeta^k\} \leq 1 \label{eq:vertex-reformulation}
\end{equation}
\end{theorem}
Unfortunately, the number $N$ of vertices of a polyhedron may grow exponentially with the uncertainty set dimension so that \eqref{eq:vertex-reformulation} becomes computationally burdensome with increasing uncertainty dimension.

\subsubsection{Approximate reformulations}
In the general case ($G$, $\Uu$ not polyhedral, large uncertainty dimension, etc.), we have to resort to tractable approximations instead of exact reformulations.
 A few safe approximations are available in the robust optimization literature depending on the level of hypothesis on $g$ or $\Uu$.
  We use essentially two of them, proposed by \citet{bertsimas2006tractable} and \citet{roos2018approximation}.

\begin{theorem}[Bertsimas and Sim approximation]\label{th:Bertsimas-approximation}
If $\Uu$ is an uncertainty set defined by an absolute norm (e.g. those of section \ref{sec:uncertainty-sets}), 
 the robust constraint \eqref{eq:generic-robust-strength-constraint} can be safely approximated as follows \citep{bertsimas2006tractable}:
\begin{equation}
g(\bsig) + \|\bs\|_* \leq 1 \label{eq:criterion-BS}
\end{equation}
where for $j=1,\ldots,m$:
\begin{equation}
s_j=\max\{g(\bSig_j)+b_j,g(-\bSig_j)-b_j\} \label{eq:BS-s-definition}
\end{equation}
with $\bSig_j$ denoting the $j$-th column of $\bSig$.
\end{theorem}

\begin{Remark}
Note that if $g$ is symmetric with respect to the origin, \eqref{eq:BS-s-definition} simplifies to:
\begin{equation}
s_j = g(\bSig_j) + |b_j|
\end{equation}
\end{Remark}

This approximation is extremely versatile since it relies only on very general assumptions on $g$ and $\Uu$.
 Unfortunately, it might not always be tight enough in certain cases.\\

A tighter approximation than \eqref{eq:criterion-BS} can be obtained from \citet[Th.1]{roos2018approximation} in the specific case where the uncertainty set is \textit{polyhedral}.
Details of the derivation are given in \ref{app:Roos-approximation}.
We report here only the simplified form of the corresponding approximation obtained for the box and cross-polytope uncertainty sets.

\begin{proposition}[Roos et al. approximation - box uncertainty]\label{prop:roos-box}
If $\Uu=\Uu_\infty$, the robust constraint \eqref{eq:generic-robust-strength-constraint} can be conservatively approximated using \thref{th:Roos-approximation} as follows:
\begin{equation}
\exists \bw\in\mathbb{R}^m,\: \bW\in \mathbb{R}^{d\times m} \st
\begin{cases} \displaystyle{\sum_{j=1}^m w_j + g\left(\bsig - \sum_{j=1}^m \bW_j\right) \leq 1} \\
 g(\bW_j - \bSig_j) -b_j \leq w_j \quad \forall j=1,\ldots,m\\
 g(\bW_j + \bSig_j) + b_j \leq w_j
\end{cases}
\end{equation}\label{eq:Roos-approx-box}
\end{proposition}

\begin{proposition}[Roos et al. approximation - cross-polytope uncertainty]
If $\Uu=\Uu_1$, the robust constraint \eqref{eq:generic-robust-strength-constraint} can be conservatively approximated using \thref{th:Roos-approximation} as follows:
\begin{equation}
\exists w\in\mathbb{R},\: \bW\in \mathbb{R}^{d} \st
\begin{cases} \displaystyle{w + g\left(\bsig - \bW\right) \leq 1} \\
 g(\bW - \bSig_j) -b_j \leq w \quad \forall j=1,\ldots,m\\
 g(\bW + \bSig_j) + b_j \leq w
\end{cases}
\end{equation}\label{eq:Roos-approx-cross}
\end{proposition}

\todo{With $\bW=\bsig$, we get $w\leq 1$ and $g(\bsig \pm \bSig_j) \pm b_j \leq 1$ which is the vertex-based formulation.
Hence Roos et al. approximation is exact in this case.
Can we explain why it is exact for this case and not for the box uncertainty ? Is there a link with the fact that affine decision rules are exact for simplices ?}

\begin{Remark}
Interestingly both approximations obtained from \thref{th:Bertsimas-approximation} and \thref{th:Roos-approximation} can be formulated using conic constraints of the same form as those involved in the definition of $g$.
The main difference with the nominal case regarding computational complexity is that the resulting robust constraint involves a finite number of instances of $g$ instead of just one e.g. $2m+1$ instances for \eqref{eq:Roos-approx-box} and \eqref{eq:Roos-approx-cross}.
\end{Remark}

\subsubsection{Illustrative example}
To illustrate the effectiveness of the previous approximate reformulations, we consider the determination of the eroded strength condition in the case of a plane stress von Mises criterion given by:
\begin{equation}
\bsig \in  G \quad \Leftrightarrow\quad g(\bsig) = \frac{1}{\sigma_0}\sqrt{\sigma_1^2+\sigma_2^2-\sigma_1\sigma_2}\leq 1
\end{equation}
where $\sigma_1,\sigma_2$ are the principal stresses and $\sigma_0$ the material uniaxial strength.

We determine the corresponding eroded strength domain:
\begin{equation}
\bsig \in  G\ominus\Uu_\Sigma \quad \Leftrightarrow\quad g(\bsig + \bsig_\zeta) \leq 1 \quad \forall \bsig_\zeta \in \Uu_\Sigma
\end{equation}
when $\Uu_\Sigma \subseteq \RR^2$ is given by, either:
\begin{equation}
\bsig_\zeta \in \Uu_\Sigma \quad \Leftrightarrow\quad |\bsig_{\zeta,1}| + |\bsig_{\zeta,2}| \leq \alpha\sigma_0
\end{equation}
or:
\begin{equation}
\bsig_\zeta \in \Uu_\Sigma \quad \Leftrightarrow\quad \max\{|\bsig_{\zeta,1}|; |\bsig_{\zeta,2}|\} \leq \alpha\sigma_0
\end{equation}
with $0<\alpha<1$. The above stress uncertainty set $\Uu_\Sigma$ corresponds to a scaled cross-polytope or box uncertainty, that is $\Uu_\Sigma = \alpha\sigma_0 \Uu_1$ or $\Uu_\Sigma = \alpha\sigma_0 \Uu_\infty$, respectively.

Figure \ref{fig:eroded-domain-example} represents the obtained eroded domains in both cases for $\alpha=0.25$.
It can clearly be seen that the approximation obtained from \thref{th:Bertsimas-approximation} is much more conservative than that obtained from \thref{th:Roos-approximation}. The latter is in fact of very good quality as it seems to be exact for the cross-polytope uncertainty (Figure \ref{fig:eroded-L1}) whereas it is very close to the exact domain obtained with the vertex-based representation of \thref{th:vertex-reformulation} in the box uncertainty case (Figure \ref{fig:eroded-Linf}).
\begin{figure}
\begin{center}
\begin{subfigure}{0.49\textwidth}
\includegraphics[width=\textwidth]{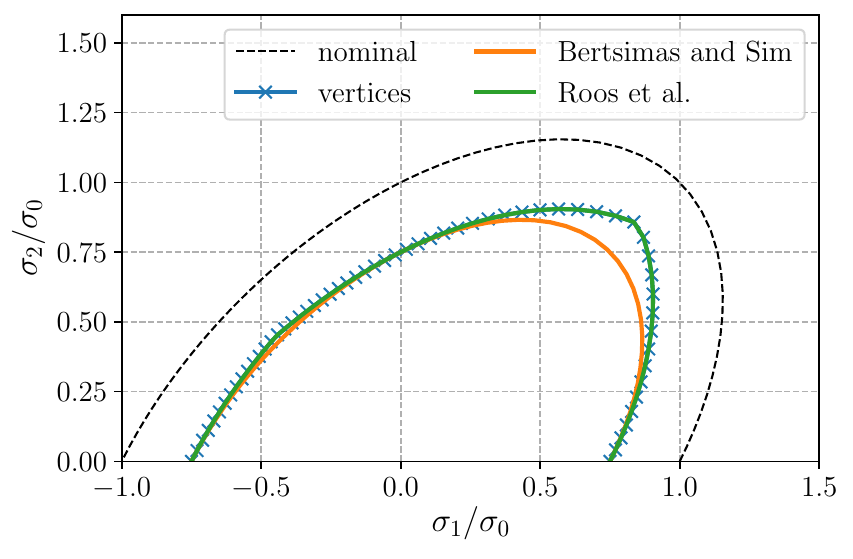}
\caption{$\Uu_\Sigma = \alpha\sigma_0 \Uu_1$}\label{fig:eroded-L1}
\end{subfigure}
\hfill
\begin{subfigure}{0.49\textwidth}
\includegraphics[width=\textwidth]{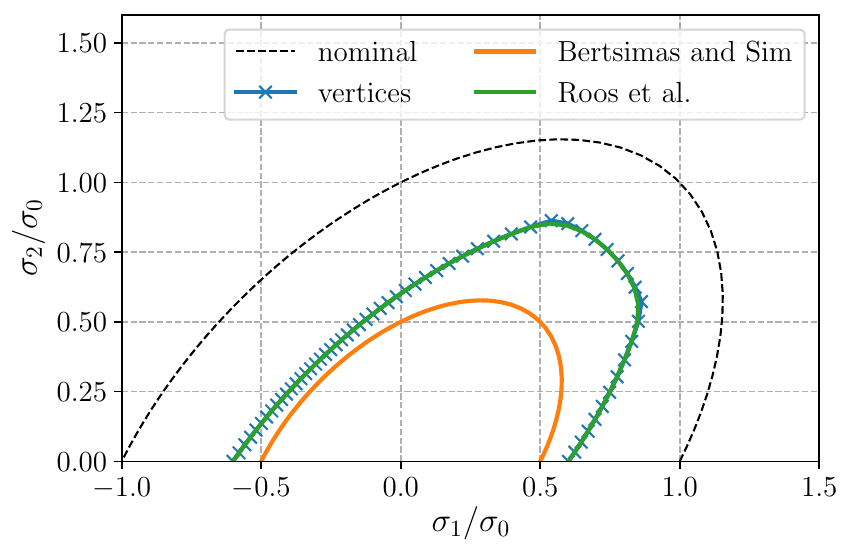}
\caption{$\Uu_\Sigma = \alpha\sigma_0 \Uu_\infty$}\label{fig:eroded-Linf}
\end{subfigure}
\end{center}
\caption{Eroded strength domains in the case $\alpha=0.25$. The exact domain is obtained using the vertex-based representation of \protect\thref{th:vertex-reformulation} ("vertices") and compared with the approximations based on \protect\thref{th:Bertsimas-approximation} ("Bertsimas and Sim") and \protect\thref{th:Roos-approximation} ("Roos et al.").}
\label{fig:eroded-domain-example}
\end{figure}

\subsection{General tractable formulations in case of small uncertainty}
Let us now depart from homothetic uncertainty assumption \eqref{eq:homothetic-uncertainty} and consider instead $\bk$ to be a set of parameters characterizing the strength criterion shape e.g. $\bk=(c,\phi)$ for a Mohr-Coulomb criterion with cohesion $c$ and internal friction angle $\phi$.
Let us assume that there exist a matrix $\bA(\bk)$ and a scalar $b(\bk)$ depending on $\bk$ and a homogeneous convex function $h$ such that the strength condition can be written as follows\footnote{Note that $\frac{1}{b} h\circ \bA$ is the gauge function $g_G$ of $G$.}:
\begin{equation}
\bsig \in G(\bk) \quad \Longleftrightarrow \quad h(\bA(\bk)\bsig)\leq b(\bk)
\end{equation}

\begin{Remark}
Note that if the strength criterion can be written using more than one constraint of the previous form, the following reformulation can be applied to all such constraints independently.\\
\end{Remark}

Let us now consider an uncertainty on the parameters of the form $\bk = \bk_0 + \bK \bzeta$ with $\bzeta\in \mathcal{U}$  as in section \ref{sec:physical-uncertainty} and assume that the uncertainty is small i.e. $\|\bK \bzeta\|\ll \|\bk_0\|$. The robust counterpart of the uncertain strength criterion then writes as:
\begin{equation}
 h(\bA(\bk_0 + \bK \bzeta)\bsig)\leq b(\bk_0 + \bK \bzeta) \quad \forall \bzeta\in \Uu
\end{equation}

This expression may not be reformulated in a tractable manner for general $\bA$ and $b$. We then propose to obtain a tractable approximation using a first-order Taylor expansion accounting for the hypothesis of small uncertainty:
\begin{align}
 h((\bA(\bk_0) + \bA'(\bk_0)\bK \bzeta)\bsig) &\leq b(\bk_0) +b'(\bk_0)\bK \bzeta \quad \forall \bzeta\in \Uu 
 \label{eq:strength-robust-approximation}
\end{align}
Let us remark that the previous approximation is in fact exact if $\bA$ and $b$ depend linearly upon $\bk$. \eqref{eq:strength-robust-approximation} is also  equivalent to:
\begin{align}
\sup_{\bzeta \in \Uu} \{h((\bA(\bk_0) + \bA'(\bk_0)\bK \bzeta)\bsig) - b(\bk_0) - b'(\bk_0)\bK \bzeta\} \leq 0
\end{align}

Unfortunately, this maximization is untractable in practice but can be conservatively using \thref{th:Bertsimas-approximation} as follows:
\begin{equation}
h(\bA(\bk_0)\bsig)-b(\bk_0) + \|\bs\|_{*} \leq 0 \label{eq:Bertsimas-approximation-physical}
\end{equation}
where $\bs=(s_j)_{j=1,\ldots,m}$ and:
\begin{align}
s_j &= \max\{h(\Delta \bA_j \bsig) - \Delta b_j,h(-\Delta \bA_j \bsig) + \Delta b_j\}\\ \label{BS-s-definition}
\Delta \bA_j &= \bA'(\bk_0)\bK_j\\
\Delta b_j &= b'(\bk_0)\bK_j
\end{align}
in which $\bK_j$ denotes the $j$-th column of $\bK$.

In case of right-hand side uncertainty only $\bA(\bk)=\bA(\bk_0)$, $\Delta\bA_j=0$ so that \eqref{eq:Bertsimas-approximation-physical} becomes:
\begin{equation}
h(\bA(\bk_0)\bsig)-b(\bk_0) + \|b'(\bk_0)\bK\|_* \leq 0
\end{equation}
which coincides with \eqref{eq:homothetic-RC} when both $b(\bk)$ and $\beta(\bzeta)$ are affine.\\

In conclusion, we propose to consider the following conservative robust counterpart $G_\text{RC}(\bk,\bK)$:
\begin{equation}
\bsig \in G_\text{RC}(\bk,\bK) \quad \Longleftrightarrow \quad h(\bA(\bk_0)\bsig) + \|\bs\|_{*} \leq b(\bk_0)
\end{equation}
which enjoys the following properties:
\begin{itemize}
\item safe approximation to the robust criterion:
\begin{equation}
\bsig \in G_\text{RC}(\bk,\bK) \quad \Longrightarrow \quad \bsig \text{ satisfies \eqref{eq:strength-robust-approximation}}
\end{equation}
\item consistency with the nominal criterion:
\begin{equation}
 G_\text{RC}(\bk,\bK=\boldsymbol{0}) = G(\bk_0)
 \end{equation}
\item similar modeling complexity as the nominal criterion. Indeed, if $G$ is SOCP or SDP-representable, then for the considered uncertainty sets of section \ref{sec:uncertainty-sets}, $G_\text{RC}$ is also  SOCP or SDP-representable.
Similarly, if $G$ is LP-representable, so is  $G_\text{RC}$ except for $\Uu=\Uu_2$ for which it is SOCP-representable.
Since the number of material parameters $|\bk|$ is usually small, the additional modeling cost for representing $G_\text{RC}$ is moderate.
\end{itemize}

\subsection{Application to a robust Mohr-Coulomb criterion}
For the sake of illustration, let us consider the case of a Mohr-Coulomb strength criterion with tension cut-off where the cohesion $c$, the friction angle $\phi$ and the tension cut-off $f_t$ are all uncertain i.e. let $\bk=(c,\phi, f_t)$ and represent their uncertainty using \eqref{eq:physical-uncertainty}.
For illustrative purposes, we assume that the tension cut-off is independent of $c$ and $\phi$.
Regarding the cohesion and friction angle, a negative correlation is often encountered between both parameters, i.e. soils with low cohesion tend to exhibit higher friction angles than with higher cohesion.
We denote by $\rho$ the correlation coefficient between $c$ and $\phi$, with typical values ranging from $-0.5$ to $-0.9$ \citep{wang2016quantifying}.
We therefore consider a "correlation" matrix $\bK$ such that:
\begin{equation}
\bK\T\bK = \begin{bmatrix}
\Delta c^2 & \rho \Delta c \Delta\phi  & 0 \\
 \rho \Delta c \Delta\phi  & \Delta\phi^2 & 0\\
 0 & 0 & \Delta f_t^2 
\end{bmatrix} \quad \text{i.e. } \bK = \begin{bmatrix}
\Delta c & \rho \Delta\phi  & 0 \\
0 & \Delta\phi\sqrt{1-\rho^2}  &  0\\
 0 & 0 & \Delta f_t 
\end{bmatrix} \label{c-phi-correlation}
\end{equation}
where $\Delta c,\Delta\phi,\Delta f_t$ are the parameters typical variations and are assumed to be positive.
Note that if such variations were taken as the standard deviations of the corresponding parameters, $\bK\T\bK$ would be the corresponding covariance matrix.\\

The robust counterpart of the Coulomb criterion therefore reads:
\begin{equation}
\sigma_{1}-\sigma_3 + (\sigma_{1}+\sigma_3)\sin\phi(\bzeta) -2c(\bzeta)\cos\phi(\bzeta) \leq 0 \qquad \forall\bzeta \in\Uu 
\end{equation}
\begin{align}
\sigma_{1}-\sigma_3 + (\sigma_{1}+\sigma_3)\sin\phi(\bzeta) \leq 2c(\bzeta)\cos\phi(\bzeta) & \qquad \forall\bzeta \in\Uu \\
\sigma_3 \leq f_t(\bzeta) \qquad \quad\quad & \qquad \forall \bzeta \in \Uu
\end{align}
where $\sigma_1$ (resp. $\sigma_3$) is the maximum (resp. minimum) principal stress and $\Uu$ is one of the typical uncertainty sets of section \ref{sec:uncertainty-sets}.\\

First, the tension cut-off part being independent of $(c,\phi)$, it can be treated as in section \ref{sec:homothetic-uncertainty}, yielding:
\begin{align}
\sigma_3 &\leq f_{t0}+\Delta f_t \zeta_3 \quad \forall \bzeta \in \Uu \\
\Leftrightarrow \quad \sigma_3 &\leq f_{t0}-\Delta f_t
\end{align}
where $f_{t0}$ is tensile strength nominal value.
As a result, the robust counterpart of the tension cut-off criterion amounts to considering the worst-case tensile strength $f_{t0}-\Delta f_t$ as one could have expected.\\

Regarding the Mohr-Coulomb part, linearization around $\bk_0=(c_0,\phi_0, f_{t0})$ results in:
\begin{align}
\sigma_{1}-\sigma_3 &+ (\sigma_{1}+\sigma_3)(\sin\phi_0 +\cos(\phi_0)K_{22}\zeta_2) \notag\\
 &-2(c_0+K_{11}\zeta_1 + K_{12}\zeta_2)\cos\phi_0 \notag \\
 &+ 2c_0\sin\phi_0K_{22}\zeta_2 \leq 0 \qquad\qquad\qquad \forall\bzeta \in\Uu 
\end{align}
with $K_{ij}$ being the components of $\bK$ defined in \eqref{c-phi-correlation}.
This yields the following robust counterpart:
\begin{equation}
\sigma_{1}-\sigma_3 + (\sigma_{1}+\sigma_3)\sin\phi_0 -2c_0\cos\phi_0 + \|\bs\|_* \leq 0 \label{eq:robust-Mohr-Coulomb}
\end{equation}
where:
\begin{equation}
\bs = \begin{pmatrix}
2\Delta c\cos\phi_0 \\
\left((\sigma_{1}+\sigma_3)\cos(\phi_0)+2c_0\sin\phi_0\right)\sqrt{1-\rho^2}\Delta\phi-2c_0\cos\phi_0\rho\Delta\phi\\
0
\end{pmatrix}
\end{equation}\\

Let us now investigate the simple case of no cross-correlation $\rho=0$ with a box uncertainty $\Uu=\Uu_\infty$.
The previous expression reduces to:
\begin{align}
\bs &= \begin{pmatrix}
2\Delta c\cos\phi_0 \\
\left((\sigma_{1}+\sigma_3)\cos(\phi_0)+2c_0\sin\phi_0\right)\Delta\phi\\
0
\end{pmatrix} \\
\|\bs\|_*=\|\bs\|_1 &= 2\Delta c\cos\phi_0 + \left|(\sigma_{1}+\sigma_3)\cos(\phi_0)+2c_0\sin\phi_0\right|\Delta\phi
\end{align}
so that the robust Mohr-Coulomb criterion \eqref{eq:robust-Mohr-Coulomb} reduces to:
\begin{align}
\sigma_{1}-\sigma_3 + (\sigma_{1}+\sigma_3)\sin\phi_0 & \notag \\
+ \left|(\sigma_{1}+\sigma_3)\cos(\phi_0)+2c_0\sin\phi_0\right|\Delta\phi &\leq 2(c_0-\Delta c)\cos\phi_0
\end{align}
which can be further expressed as follows:
\begin{equation}
\begin{cases}
\sigma_{1}-\sigma_3 + (\sigma_{1}+\sigma_3)(\sin\phi_0 + \cos(\phi_0)\Delta\phi) \leq 2c_\text{min}\cos\phi_0 - 2c_0\sin\phi_0\Delta\phi \\
\sigma_{1}-\sigma_3 + (\sigma_{1}+\sigma_3)(\sin\phi_0 - \cos(\phi_0)\Delta\phi) \leq 2c_\text{min}\cos\phi_0 + 2c_0\sin\phi_0\Delta\phi
\end{cases}
\end{equation}
where $c_\text{min}=c_0-\Delta c$ is the worst-case cohesion.
Introducing $\phi_\text{min}=\phi_0-\Delta \phi$ the worst-case friction angle and $\phi_\text{max}=\phi_0+\Delta \phi$ the best-case friction angle and using the fact that $\sin(\phi_\text{max/min}) \approx \sin\phi_0 \pm \cos(\phi_0)\Delta\phi$ and $\cos(\phi_\text{max/min}) \approx \cos\phi_0 \mp \sin(\phi_0)\Delta\phi$, the previous criterion is, in fact, a first-order approximation (in terms of $\Delta c,\Delta\phi$) to the following multi-surface criterion:
\begin{equation}
\begin{cases}
\sigma_{1}-\sigma_3 + (\sigma_{1}+\sigma_3)\sin\phi_\text{max}  \leq 2c_\text{min}\cos(\phi_\text{max})\\
\sigma_{1}-\sigma_3 + (\sigma_{1}+\sigma_3)\sin\phi_\text{min}  \leq 2c_\text{min}\cos(\phi_\text{min})
\end{cases}
\end{equation}
i.e. the obtained robust counterpart, for this specific case, (approximately) corresponds to the intersection of two Coulomb criteria with the worst-case cohesion and either the best or the worst-case friction angle. 

Figure \ref{fig:robust-MC} represents the corresponding nominal and robust version for this uncorrelated box-uncertainty case.
The yield surface corresponding to random realizations of $c(\bzeta)$ and $\phi(\bzeta)$ are also represented.
One can indeed see that the obtained robust strength criterion forms a tight lower bound to the various realizations and is made of two sets of lines approximately characterized by the minimum and maximum friction angle $\phi_\text{min}$ and $\phi_\text{max}$.

To assess the influence of the uncertainty set, the obtained robust strength criterion boundary has been represented in Figure \ref{fig:robust-Coulomb-comp-U} in Mohr's normal/shear stress plane for the case of the box $\Uu_\infty$, ellipsoid $\Uu_2$ and cross-polytope $\Uu_1$ uncertainty.
As expected, the most conservative robust criterion is obtained for the box, then follows the ellipsoid, then the cross-polytope uncertainty.
 The latter is made of 3 different planes whereas the ellipsoid robust criterion has a curved shape upon close inspection.
 Interestingly, both of them are very close to each other.
 All three criteria result in the same cohesion value for pure shear ($\tau=c_\text{min}$) and the same asymptotic slope for large compressions (approximately corresponding to a friction angle $\phi_\text{min}$).
  The corresponding tensile strength is however different for the three criteria.

Finally, the influence of the cross-correlation parameter $\rho$ on the criterion shape is also represented on Figure \ref{fig:robust-Coulomb-comp-rho}.
 Its effect is non-negligible since, between $\rho=0$ and $\rho=-0
5$, the robust tensile strength decreases by roughly 5\% and the shear strength to increase by roughly 10\%.

\begin{figure}
\begin{center}
\includegraphics[width=0.6\textwidth]{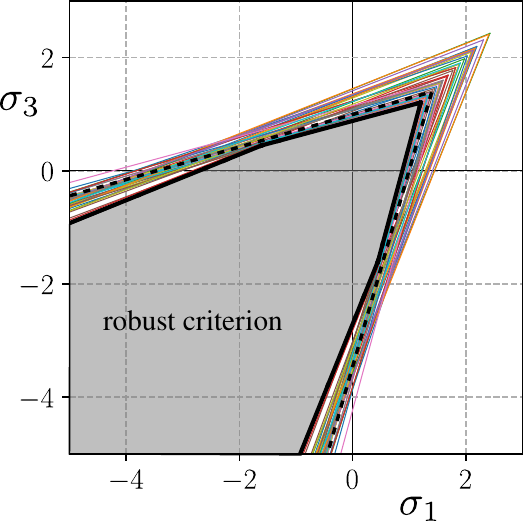}
\end{center}
\caption{Robust and uncertain Mohr-Coulomb criterion : $c_0=1$ MPa, $\phi_0=30^\circ$, $\Delta c = 150$ kPa, $\Delta \phi = 5^{\circ}$. 
Black dashed lines denote the nominal surface, thin coloured lines denote random realizations of the uncertain criterion.}\label{fig:robust-MC}
\end{figure}

\begin{figure}
\begin{center}
\includegraphics[width=0.6\textwidth]{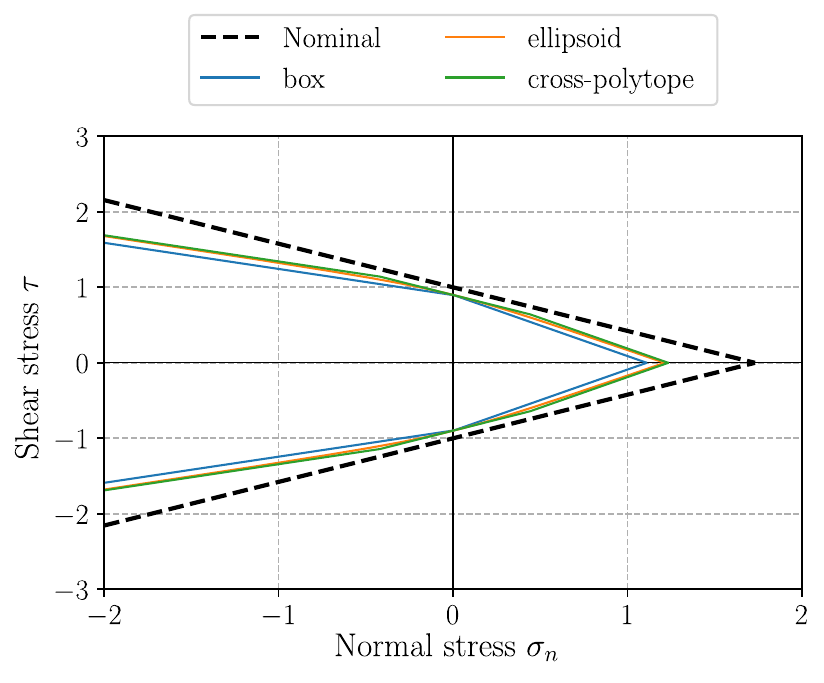}
\end{center}
\caption{Influence of the chosen uncertainty set on the obtained robust criterion.}\label{fig:robust-Coulomb-comp-U}
\end{figure}

\begin{figure}
\begin{center}
\includegraphics[width=0.6\textwidth]{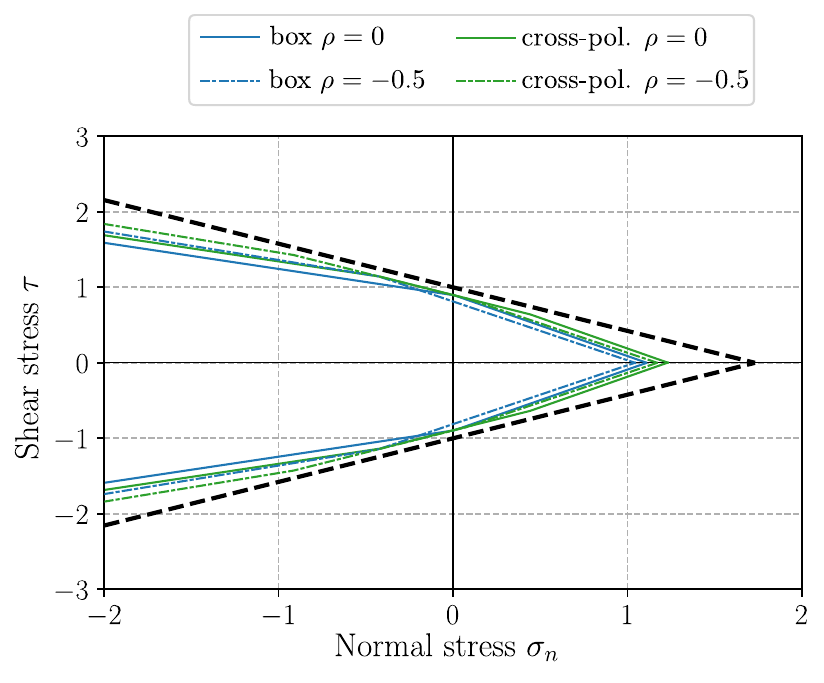}
\end{center}
\caption{Influence of the cross-correlation $\rho$ on the obtained robust criterion.}\label{fig:robust-Coulomb-comp-rho}
\end{figure}

\section{Affinely adjustable robust counterparts for strength uncertainties}\label{sec:AARC-main}
In the previous section, we have been able to derive various robust counterparts for strength uncertainties under fairly general modeling hypotheses.
 This was essentially possible since we restricted to static decision rules for the stress field.
 In our opinion, such rules are quite efficient, that is to say they are not too conservative, in situations where strength uncertainties are of small amplitude so that the resulting stress field is only mildly perturbed in presence of uncertainties.

In the following sections, we will discuss the case where such static rules may fail and derive tractable (approximate) reformulations of robust strength conditions in the AARC case discussed in section \ref{sec:AARC-formulation}. 

\subsection{Typical examples where RC formulations are too conservative}\label{sec:static-failure}

RC formulations as derived in section \ref{sec:RC-strength-uncertainty} can be typically used in situations where uncertainties on strength material properties are small e.g. spatial variability with slow amplitude or small uncertainty on experimentally measured values.
 Robust optimization can also be used to assess vulnerability of a given system with respect to the partial or total loss of some of its components.
  For instance, let us consider a framed building made of beams and columns with given nominal strengths.
  One could wonder how robust would be such a structure to the loss of, say, $\Gamma$ columns or beams.
  This problem could be tackled by resorting to a binary optimization limit analysis problem in which each beam/column strength condition, indexed by $i=1,\ldots,N$, would be weighted by a binary coefficient $z_i\in \{0,1\}$ where 0 means that the member is sound and 1 that it is damaged and such that the maximum number of damaged elements is bounded:  $\sum_{i=1}^N z_i \leq \Gamma$.
  To fix ideas and simplify the discussion, let us write such a constraint using a scalar stress $\sigma_i$ and its corresponding nominal strength $\sigma_0$ as follows:
\begin{equation}
|\sigma_i| \leq (1-\eta z_i)\sigma_0 \label{eq:component-loss}
\end{equation}
where $0\leq \eta\leq 1$ is the maximum degradation level ($\eta=1$ means complete damage).

 However, such binary optimization problems are notoriously hard to solve in practice, especially for large-scale problems.
  In their seminal paper \citet{bertsimas2004price} considered the budget uncertainty set \eqref{eq:budget-uncertainty} to account for the deviation of at most $\Gamma$ coefficients from their nominal value in an uncertain linear optimization problem.
  This linear programming approach can be seen as a relaxation of binary optimization problems such as the one we just mentioned.
  Under some particular conditions, the relaxed LP formulation can be in fact exact.
  Following this idea, we could therefore consider a continuous variant of the previous problem where we introduce:
\begin{align}
\bzeta&=(\zeta_i)_{i=1,\ldots,N}\\
\Uu^+(\Gamma)&=\left\lbrace\bzeta \st 0 \leq \zeta_i \leq 1 \: \forall i=1,\ldots, N \text{ and } \sum_{i=1}^N \zeta_i \leq \Gamma\right\rbrace
\end{align} 
and consider the same uncertain strength condition \eqref{eq:component-loss} by replacing the binary variable $z_i$ with its continuous version $\zeta_i$ with $\bzeta \in \Uu^+(\Gamma)$ where $\Uu^+(\Gamma)$ is a one-sided budget uncertainty set (see also \ref{app:dual-norms}).\\

Following \eqref{eq:static-RO-LA-1}, the worst-case load factor with a static decision rule would then be computed by finding a stress field $(\sigma_i)$ in equilibrium and such that for all $i=1,\ldots,N$:
\begin{align}
|\sigma_i| &\leq (1-\eta \zeta_i)\sigma_0 \quad \forall \bzeta\in \Uu^+(\Gamma) \\
\Leftrightarrow\quad |\sigma_i| &\leq \sigma_0 \cdot \min_{\bzeta\in\Uu^+(\Gamma)} \{1-\eta \zeta_i\}  \notag
\end{align}
However, the above minimization is given by:
\begin{equation}
 \min_{\bzeta\in\Uu^+(\Gamma)} \{1-\eta \zeta_i\} = \begin{cases} 1-\Gamma\eta & \text{if }  \Gamma \leq 1 \\
 1- \eta  & \text{if } \Gamma\geq 1 
 \end{cases}
\end{equation} 
since the objective depends only on a single $\zeta_i$ which can take the value $\min\{\Gamma;1\}$.
 Clearly, this result is way too conservative since, for $\Gamma\geq 1$, it implies that all members which can potentially fail will reach their maximum damage level $\eta$ simultaneously in the static RC counterpart.\\

The problem comes from the fact that robust optimization theory operates in a constraint-wise fashion i.e. each constraint is guarded against uncertainty independently from each other.
 Besides, no "coupling" between the constraints appears in this case since uncertain strengths are independent of each other and since we considered a static decision rule where a \textit{single} stress field $(\sigma_i)$ must be feasible for all possible uncertainty realization.
  In this case, the budget uncertainty set is treated as a box-uncertainty and the budget constraint is completely ignored as soon as $\Gamma \geq 1$.
  This is obviously too strong a requirement in practice.

\subsection{Tractable affinely adjustable robust counterparts} \label{sec:tractable-AARC}
As mentioned in section \ref{sec:AARC-formulation}, AARC formulations amount to choosing affine decision rules \eqref{eq:affine-decision-rule} for the adjustable stress and load factor variables.
 The corresponding robust counterpart \eqref{eq:AARC-LA-2} involves the following two robust constraints:
\begin{align}
\left(\displaystyle{\bsig_0 + \sum_{j=1}^m \bsig_j\zeta_j}\right) \in G(\bzeta) & \quad\forall \bzeta\in \Uu \label{eq:aarc-strength-constraint}\\
\displaystyle{\bar\lambda \leq \lambda_0 + \sum_{j=1}^m \lambda_j\zeta_j} & \quad\forall \bzeta\in \Uu
\end{align}

Introducing the vector $\bLambda = (\lambda_j)_{j=1,\ldots,m}$, the latter can be easily reformulated as follows:
\begin{align}
&\bar\lambda \leq \lambda_0 + \bLambda\T\bzeta \quad \forall\bzeta\in\Uu \notag \\
\Leftrightarrow \quad& \bar\lambda + \max_{\bzeta\in\Uu}\: \{-\bLambda\T\bzeta\} \leq \lambda_0 \label{eq:load-multiplier-reformulation}\\
\Leftrightarrow \quad & \bar\lambda + \|{-\bLambda}\|_* \leq \lambda_0\notag
\end{align} 
which is tractable for classical uncertainty sets.

The former cannot be formulated in a tractable manner in the general case.
 We now discuss some specific cases for which exact or approximate tractable reformulations of \eqref{eq:aarc-strength-constraint} can be obtained.
  In the following, we assume that the strength uncertainty is homothetic as in \eqref{eq:homothetic-uncertainty} with a linear function $\beta(\bzeta)=\bb\T\bzeta$ for the sake of simplicity. We therefore consider the following generic robust strength condition:
\begin{equation}
g\left(\bsig_0 + \sum_{j=1}^m \bsig_j\zeta_j\right) \leq 1-\bb\T\bzeta, \quad \forall \bzeta \in \Uu \label{eq:AARC-strength-condition}
\end{equation}
which is exactly the form considered in \eqref{eq:generic-robust-strength-constraint} with $\bsig=\bsig_0$ and $\bSig_j=\bsig_j$.
 We can therefore use all the reformulations or approximations discussed in section \ref{sec:generic-robust-strength-constraint} to obtain a tractable counterpart of \eqref{eq:AARC-strength-condition}.

\subsection{Illustrative example: robustness with respect to the loss of structural elements}

\subsubsection{A multifiber beam model with damageable components}
\begin{figure}
\begin{center}
\includegraphics[width=0.5\textwidth]{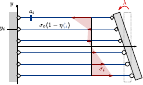}
\end{center}
\caption{Multifiber beam model under applied bending moment. Fiber uniaxial strengths are uncertain and can take values between $\sigma_0$ and $\sigma_0(1-\eta)$.}\label{fig:multifiber-pb}
\end{figure}

We consider a multifiber beam with fibers located at $y=y_i$, with $i=1,\ldots, n$ subjected to an imposed bending moment of unit reference magnitude applied via a rigid block on the right boundary (see Figure \ref{fig:multifiber-pb}). We therefore look for the maximum applicable bending moment $M^+=\lambda^+$. 
The fibers are assumed to remain horizontal and obey a uniaxial strength criterion of the form $|\sigma_i|\leq \sigma_0$ where $\sigma_i$ is the fiber axial stress and $\sigma_0$ is the fiber nominal strength.
 Assuming that each fiber is of area $a_i$, the global equilibrium can be written as:
\begin{align}
\sum_{i=1}^n -y_ia_i \sigma_i &= \lambda
\end{align}
Note that we do not specify the value of the normal force which can be left free.

The nominal limit analysis problem is therefore given by:
\begin{equation}
\begin{array}{rl}
\displaystyle{\max_{\lambda,(\sigma_i)}} & \lambda \\
\text{s.t.} & \displaystyle{\sum_{i=1}^n -y_ia_i \sigma_i = \lambda}\\
& |\sigma_i| \leq \sigma_0 \quad \forall i=1,\ldots,n
\end{array}
\end{equation}

For the uncertain case, we consider that each fiber strength is uncertain and given by \eqref{eq:component-loss}
where $\bzeta \in \Uu = \Uu_\text{budget}^+(\Gamma)$ is assumed to belong to the one-sided budget uncertainty set of dimension $n$.

The robust limit analysis problem is given by:
\begin{equation}
\begin{array}{rl}
\displaystyle{\max_{\bar\lambda}\min_{\bzeta\in\Uu}\max_{(\sigma_i(\bzeta),\lambda(\bzeta))}} & \bar\lambda \\
\text{s.t.} & \displaystyle{\sum_{i=1}^n -y_ia_i \sigma_i(\bzeta) = \lambda(\bzeta)}\\
& |\sigma_i(\bzeta)| \leq (1-\eta\zeta_i)\sigma_0 \quad \forall i=1,\ldots,n \\
& \bar\lambda \leq \lambda(\bzeta)
\end{array} \label{eq:bending-problem}
\end{equation}

We rely on a AARC formulation where $\bsig(\bzeta)=\bsig+\bSig\bzeta$ and $\lambda(\bzeta)=\lambda+\bLambda\T\bzeta$ such that:
\begin{equation}
\begin{array}{rll}
\displaystyle{\max_{\bar\lambda,\lambda,\Lambda,\bsig,\bSig}} & \bar\lambda & \\
\text{s.t.} & \displaystyle{\sum_{i=1}^n -y_ia_i \sigma_i = \lambda} & \\
& \displaystyle{\sum_{i=1}^n -y_ia_i \Sigma_{ij} = \Lambda_j} & \forall j=1,\ldots,n\\
& |\sigma_i + \be_i\T\bSig\bzeta| \leq (1-\eta\zeta_i)\sigma_0 & \forall i=1,\ldots,n,\quad \forall\bzeta\in\Uu \\
& \bar\lambda \leq \lambda + \bLambda\T\bzeta & \forall\bzeta\in\Uu
\end{array}
\end{equation}
where $\be_i$ is $i$-th canonical unit basis vector of $\RR^n$.

%

Following the results of the previous section, the corresponding AARC problem reads:
\begin{equation}
\begin{array}{rll}
\displaystyle{\max_{\bar\lambda,\lambda,\Lambda,\bsig,\bSig}} & \bar\lambda & \\
\text{s.t.} & \displaystyle{\sum_{i=1}^n -y_ia_i \sigma_i = \lambda} & \\
& \displaystyle{\sum_{i=1}^n -y_ia_i \Sigma_{ij} = \Lambda_j} & \forall j=1,\ldots,n\\
& \displaystyle{\sigma_i + \|(\bSig\T + \eta\sigma_0\bI)\be_i\|_* \leq \sigma_0}  & \forall i=1,\ldots,n \\
& \displaystyle{-\sigma_i + \|(-\bSig\T + \eta\sigma_0\bI)\be_i\|_* \leq \sigma_0}  & \forall i=1,\ldots,n \\
& \bar\lambda + \|{-\bLambda}\|_* \leq \lambda
\end{array} \label{eq:AARC-bending}
\end{equation}
Note that \eqref{eq:AARC-bending} turns out to be a simple LP problem accounting for the fact that the dual norm \eqref{eq:dual-norm-one-sided-budget} admits a simple LP representation.

\subsubsection{Results}

In the following, we assume that fibers have the same cross-section $a_i=1/n$ corresponding to a total cross-section $\sum_{i=1}^n a_i = 1$.
 We first consider the case with $n=4$ fibers.
 We numerically solve the AARC formulation with the one-sided budget uncertainty for various values of $\Gamma \in [0;n]$ and $\eta\in[0;1]$.
 Results are reported in Figure \ref{fig:bending-n4-uncertainty} in thick solid lines where we observe that $\lambda_\text{AARC}$ decreases from its nominal value at $\Gamma=0$ to $(1-\eta)\lambda_\text{N}$ for $\Gamma=n$.
 In-between, a piecewise-linear behaviour is observed corresponding to the fact that, for $\Gamma \leq 2$, the worst case corresponds to the  progressive degradation of the two outermost fibers, whereas for $\Gamma \geq 2$ the outermost fibers have reached their worst-case strength $(1-\eta)\sigma_0$ and the innermost fibers start deteriorating.
  For $\Gamma=n=4$, all fibers have reached their maximum degradation level $\eta$ and the limit load is given by $(1-\eta)\lambda_\text{N}$.

Alongside these results, we also report the evolution corresponding to the choice of a box-uncertainty with maximum value $\Gamma/n$ (dashed line in Figure \ref{fig:bending-n4-uncertainty}) instead of the budget uncertainty.
This choice can be easily seen to enforce a uniform repartition of the uncertain strength of value $\sigma_0 (1-\eta\frac{\Gamma}{n})$ for all fibers.
This yields an affine evolution of the limit load in terms of $\Gamma$: $\lambda =  (1-\eta\frac{\Gamma}{n})\lambda_\text{N}$.
In this particular case, a static decision rule would yield the same result as the AARC formulation.
 Conversely, the static decision rule for the budget uncertainty yields the conservative lower-bound dotted-dashed lines of Figure \ref{fig:bending-n4-uncertainty} corresponding to a constant value of $(1-\eta)\lambda_\text{N}$ as soon as $\Gamma \geq 1$, as discussed in section \ref{sec:static-failure}.

\begin{figure}
\begin{center}
\includegraphics[width=0.6\textwidth]{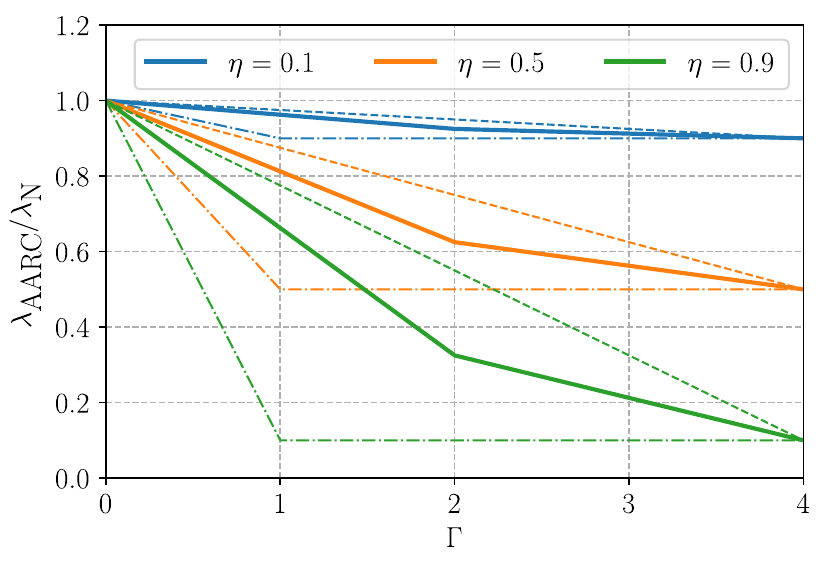}
\end{center}
\caption{Robust limit load for the bending problem with $n=4$ fibers: solid (---) lines: robust AARC solution with budget uncertainty; dashed ($--$) lines: robust AARC solution with box uncertainty of size $\Gamma/n$, dotted-dashed ($-\cdot-$) lines: robust static solution with budget uncertainty.}\label{fig:bending-n4-uncertainty}
\end{figure}

We now compare the AARC solution with budget uncertainty with random realizations of the corresponding uncertain limit analysis problem.
To do so, we draw randomly $N_s=100$ samples $\bzeta^k\in \RR^n$ and project each $\bzeta^k$ onto $\Uu^+(\Gamma)$.
We then solve the limit analysis problem corresponding to each realization and report the obtained limit load statistics (minimum, average and maximum values) in Figure \ref{fig:bending-samples} as a function of $\Gamma$.
For both $n=4$ and $n=100$ fibers, we can observe that the AARC limit load coincides with the worst-case limit load for all random samples.
 The average value corresponds to the uniform distribution where each $\bzeta_i\approx \Gamma/n$.
 For this particular example, this result holds for any value of $\eta$.
 In the case with a high number of fibers, the robust static limit load is extremely conservative since it yields $(1-\eta)\lambda_\text{N}$ for $\Gamma \geq 1$.

\begin{figure}
\begin{center}
\begin{subfigure}{0.49\textwidth}
\includegraphics[width=\textwidth]{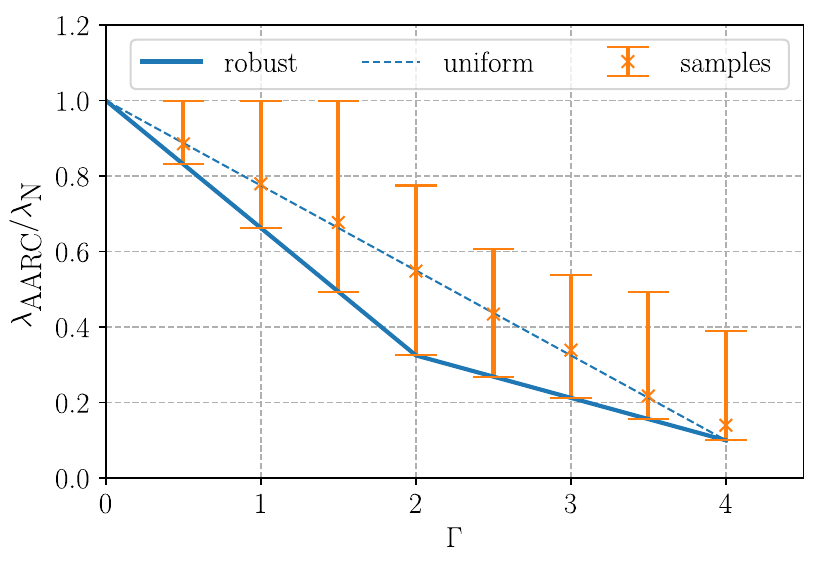}
\caption{$n=4$ fibers}
\end{subfigure}
\hfill
\begin{subfigure}{0.49\textwidth}
\includegraphics[width=\textwidth]{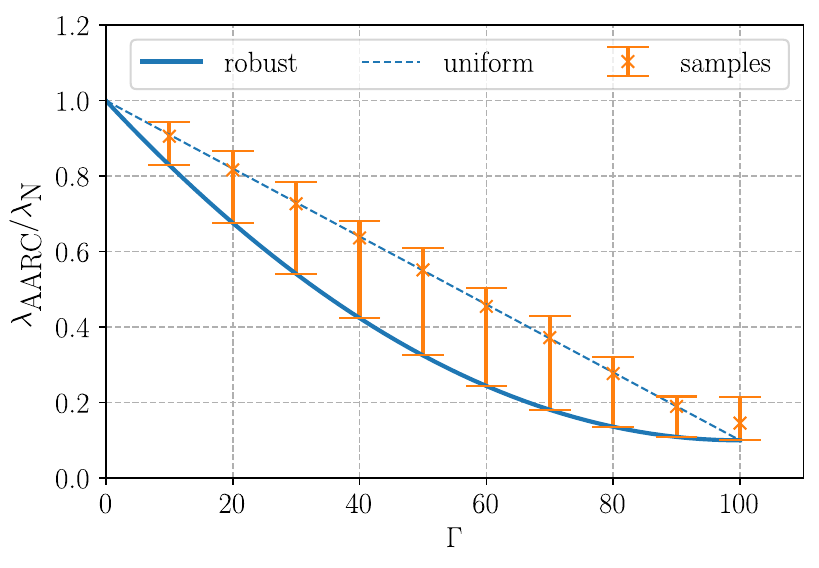}
\caption{$n=100$ fibers}\label{fig:bending-n100}
\end{subfigure}
\end{center}
\caption{Robust limit load for the bending problem with $\eta=0.9$.
 Samples are drawn randomly and projected onto $\Uu^+(\Gamma)$, error bars indicate minimum, average and maximum values.}\label{fig:bending-samples}
\end{figure}

\begin{Remark}
Note that in the limit of a large number of fibers $n\gg 1$, one can easily show that the worst-case solution is given by top and bottom 
regions with fraction $\Gamma/n$ having the worst-case strength $\sigma_0(1-\eta)$, the central region, of fraction $(1-\Gamma/n)$,
 keeping its nominal strength $\sigma_0$.
This yields the following value for the worst-case limit load:
\begin{equation}
\lambda = \left(1-\eta + \eta(1-\Gamma/n)^2\right)\lambda_\text{N}
\end{equation}
which matches very well the robust solution obtained in Figure \ref{fig:bending-n100}.
\end{Remark}

\subsubsection{Case with an additional constraint}
We now slightly change problem \eqref{eq:bending-problem} by considering the additional constraint:
\begin{equation}
\sum_{i=1}^n a_i \sigma_i(\bzeta) = 0, \label{eq:bending-problem-zero-average}
\end{equation}
enforcing a zero average force. The corresponding robust problem is similar to \eqref{eq:AARC-bending} with the additional constraints:
\begin{align}
\sum_{i=1}^n a_i \sigma_i &= 0\\
 \sum_{i=1}^n a_i \Sigma_{ij} &= 0 \quad \forall j=1,\ldots, n
\end{align}

The corresponding results for $n=100$ are reported in Figure \ref{fig:bending-zero-average}. We can see that for moderate values of $\eta$, e.g. $\eta=0.5$, the robust solution coincides with the exact worst-case solution. However, for larger values of $\eta$ the two solutions depart, with an increasing difference when $\eta \to 1$. In the particular case $\eta=1$, we observe that $\lambda_\text{AARC}=0$ as soon as $\Gamma > 1$. This example illustrates the fact that affine decision rules are not optimal in general and may also produce overly-conservative estimates. However, in this case, they are much less conservative than a static decision rule and still work well for a moderate value of $\eta$. 

\begin{figure}
\begin{center}
\begin{subfigure}{0.6\textwidth}
\includegraphics[width=\textwidth]{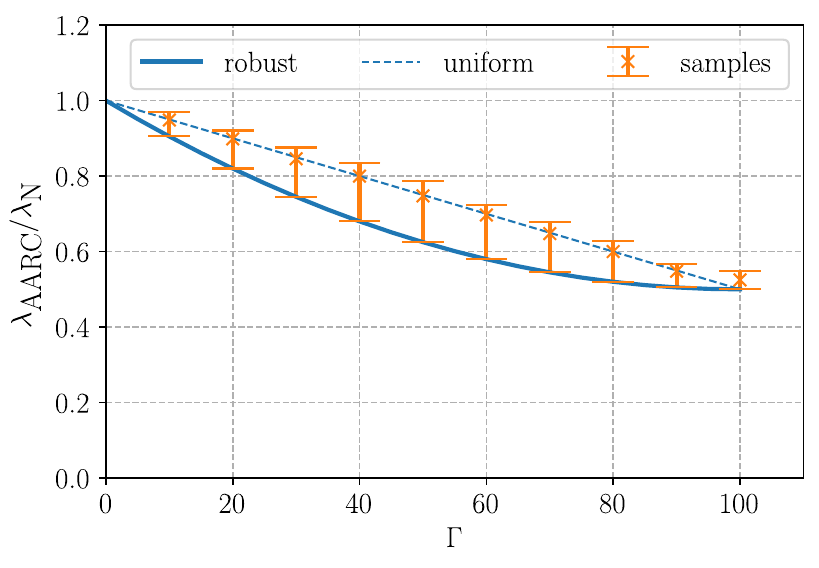}
\caption{$\eta=0.5$}
\end{subfigure}
\\
\begin{subfigure}{0.6\textwidth}
\includegraphics[width=\textwidth]{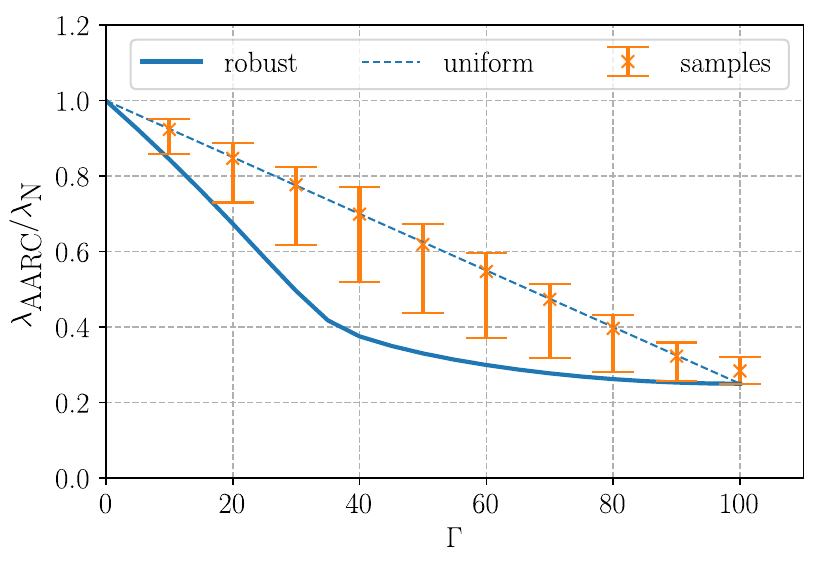}
\caption{$\eta=0.75$}
\end{subfigure}
\\
\begin{subfigure}{0.6\textwidth}
\includegraphics[width=\textwidth]{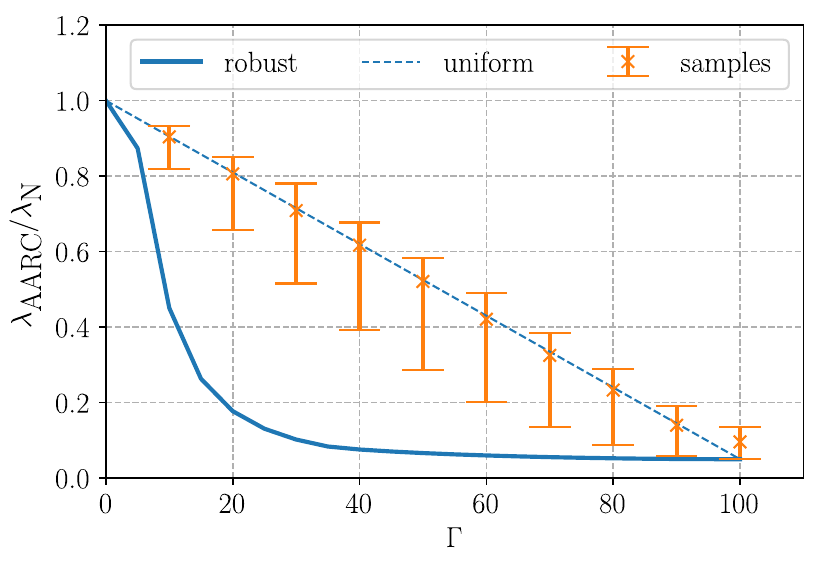}
\caption{$\eta=0.95$}
\end{subfigure}
\end{center}
\caption{Robust limit load for the bending problem \eqref{eq:bending-problem-zero-average}. Samples are drawn randomly and projected onto $\Uu^+(\Gamma)$, error bars indicate minimum, average and maximum values.} \label{fig:bending-zero-average}
\end{figure}

\section{The case of loading uncertainties}\label{sec:load-uncertainty}
Having derived the main tools needed for formulating a limit analysis problem under strength uncertainties, we can now discuss the case of loading uncertainties which use the same ideas.
 Note that we assume that strength conditions are deterministic for the sake of simplicity.
 There is however no obstacle in combining both loading and strength uncertainties.

\subsection{Uncertain limit analysis problem}

Similarly to \citet{kanno2007worst, matsuda2008robustness}, we assume here that the fixed distributed and surface loadings $\bf^\text{f}, \bt^\text{f}$ are uncertain and vary, around a nominal value, inside a convex set.
 In particular, we consider that the reference loadings $\bf^\text{r},\bt^\text{r}$ are deterministic.
 Assuming them to be uncertain adds another layer of difficulty due to the fact that the loading direction along which one has to optimize depends on the uncertainty realization.
 This specific case will be left for a future contribution.

Without loss of generality, we characterize the uncertain variation of the fixed loadings similarly to \eqref{eq:physical-uncertainty} as follows:
\begin{subequations}
\label{eq:loading-uncertainty}
\begin{align}
\bf^\text{f}(\bzeta) &= \bf^\text{f}_0 + \sum_{j=1}^m \bf^\text{f}_j \zeta_j =\bf^\text{f}_0 + \bF^\text{f} \bzeta \\
\bt^\text{f}(\bzeta) &= \bt^\text{f}_0 + \sum_{j=1}^m \bt^\text{f}_j \zeta_j = \bt^\text{f}_0 + \bT^\text{f}\bzeta 
\end{align}
\end{subequations}
where we introduced the matrices $\bF^\text{f}=[(\bf^\text{f}_j)_{j=1,\ldots,m}]$ and $\bT^\text{f}=[(\bt^\text{f}_j)_{j=1,\ldots,m}]$
and where $\bzeta \in \Uu$ with $\Uu$ a given convex uncertainty set. The corresponding uncertain limit analysis problem therefore reads:
\begin{equation}
\begin{array}{rll}
\displaystyle{\lambda^+(\bzeta) =\max_{\lambda, \bsig}} & \lambda \\
\text{s.t.} & \div\bsig + \lambda \bf^\text{r} + \bf^\text{f}_0 + \bF^\text{f} \bzeta= 0 \\
& \bsig\cdot\bn = \lambda \bt^\text{r} + \bt^\text{f}_0 + \bT^\text{f}\bzeta   \\
& \bsig \in G
\end{array} \label{eq:uncertain-LA-loading}
\end{equation}

\subsection{Robust counterpart}
Clearly, for this load uncertainty case, the use of static decision rules is doomed to fail since one cannot expect finding, except in very
 specific cases, a single stress field which is statically admissible with any realization of the uncertain loading \eqref{eq:loading-uncertainty}.
 One must therefore resort to an adjustable robust optimization which, similarly to \eqref{eq:adjustable-LA-1}, reads:
\begin{equation}
\begin{array}{rll}
\displaystyle{\lambda_\text{ARC} =\max_{\bar\lambda}} & \bar\lambda & \\
\text{s.t.} & \forall \bzeta \in \Uu, \exists \bsig,\lambda \st& \div\bsig + \lambda \bf^\text{r} + \bf^\text{f}_0 + \bF^\text{f} \bzeta = 0\\
& & \bsig\cdot\bn = \lambda \bt^\text{r} + \bt^\text{f}_0 + \bT^\text{f}\bzeta    \\
& & \bsig \in G \\
& & \bar\lambda \leq \lambda
\end{array} \label{eq:adjustable-LA-loading}
\end{equation}

Again, in order to obtain a safe and tractable approximation to the above robust formulation, we resort to the use of the affine decision rules \eqref{eq:affine-decision-rule} and obtain the following AARC:
\begin{equation}
\begin{array}{rll}
\displaystyle{\lambda_\text{AARC} =\max_{\bsig_i,\lambda_i}}\min_{\bzeta\in\Uu} & \displaystyle{\lambda_0 + \sum_{j=1}^m \lambda_j\zeta_j}\\
\text{s.t.} &\displaystyle{ \div\left(\bsig_0 + \sum_{j=1}^m \bsig_j\zeta_j\right)+ \left(\lambda_0 + \sum_{j=1}^m \lambda_j\zeta_j\right) \bf^\text{r} +  \bf^\text{f}_0 + \bF^\text{f} \bzeta= 0 } \\
& \displaystyle{\left(\bsig_0 + \sum_{j=1}^m \bsig_j\zeta_j\right)\cdot\bn = \left(\lambda_0 + \sum_{j=1}^m \lambda_j\zeta_j\right)\bt^\text{r} + \bt^\text{f}_0 + \bT^\text{f}\bzeta   }\\
& \left(\displaystyle{\bsig_0 + \sum_{j=1}^m \bsig_j\zeta_j}\right) \in G
\end{array} \label{eq:AARC-LA-loading-1}
\end{equation}

which can further be formulated as follows:
\begin{equation}
\begin{array}{rlll}
\displaystyle{\lambda_\text{AARC} =\max_{\bar\lambda, \bsig_i,\lambda_i}} &\bar\lambda&\\
\text{s.t.} &\displaystyle{ \div(\bsig_i) + \lambda_i \bf^\text{r} + \bf^\text{f}_i = 0 }& \forall i=0,\ldots,m \\
& \displaystyle{\bsig_i\cdot\bn = \lambda_i\bt^\text{r}  + \bt^\text{f}_i} & \forall i=0,\ldots,m\\
& \left(\displaystyle{\bsig_0 + \sum_{j=1}^m \bsig_j\zeta_j}\right) \in G & \forall \bzeta\in \Uu\\
& \displaystyle{\bar\lambda \leq \lambda_0 + \sum_{j=1}^m \lambda_j\zeta_j} & \forall \bzeta\in \Uu
\end{array} \label{eq:AARC-LA-loading}
\end{equation}

Clearly, \eqref{eq:AARC-LA-loading} bears striking similarities with \eqref{eq:AARC-LA-2} in the sense that we look for $1+m$ stress fields statically admissible with a given loading (here we have an additional fixed loading for each $j=1,\ldots,m$ compared to \eqref{eq:AARC-LA-2}).
In particular, uncertainty has been removed from the equilibrium equations whereas only the last two constraints are robust ones which must be reformulated.
Both of these constraints are exactly those which have already been considered in section \ref{sec:tractable-AARC}.
In particular, the robust strength constraint can be reformulated, either exactly or approximately, using the results of section \ref{sec:generic-robust-strength-constraint}.

For instance, if the loading uncertainty set is the box uncertainty $\Uu = \Uu_\infty$, using \eqref{eq:load-multiplier-reformulation} and the result of \thref{prop:roos-box}, the corresponding tractable AARC reads:
\begin{equation}
\begin{array}{rlll}
\displaystyle{ \max_{\bar\lambda, \bsig_i,\lambda_i, w_j, \bW_j}} &\bar\lambda&\\
\text{s.t.} &\displaystyle{ \div(\bsig_i) + \lambda_i \bf^\text{r} + \bf^\text{f}_i = 0 }& \forall i=0,\ldots,m \\
& \displaystyle{\bsig_i\cdot\bn = \lambda_i\bt^\text{r}  + \bt^\text{f}_i} & \forall i=0,\ldots,m\\
& \displaystyle{\sum_{j=1}^m w_j + g\left(\bsig_0 - \sum_{j=1}^m \bW_j\right) \leq 1 } & \\
& g(\bW_j - \bsig_j) \leq w_j & \forall j=1,\ldots,m \\ 
& g(\bW_j + \bsig_j) \leq w_j & \forall j=1,\ldots,m \\
& \displaystyle{\bar\lambda + \sum_{j=1}^m |\lambda_j| \leq \lambda_0} & 
\end{array}  \label{eq:AARC-LA-loading-example}
\end{equation}
Note that since \eqref{eq:AARC-LA-loading-example} uses the result from \thref{prop:roos-box}, we only obtain a safe approximation for $\lambda_\text{AARC}$ solution to \eqref{eq:AARC-LA-loading}.
 Since the uncertainty is polyhedral, we could have used instead the exact vertex-based reformulation of \thref{th:vertex-reformulation}.
 However, the latter would involve $2^m$ strength constraints whereas the approximate formulation \eqref{eq:AARC-LA-loading-example} involves only $1+2m$ such constraints which is much less burdensome if the uncertainty dimension $m$ is large.

\subsection{Limit analysis of a truss structure under loading uncertainties}

\begin{figure}
\begin{center}
\includegraphics[width=0.8\textwidth]{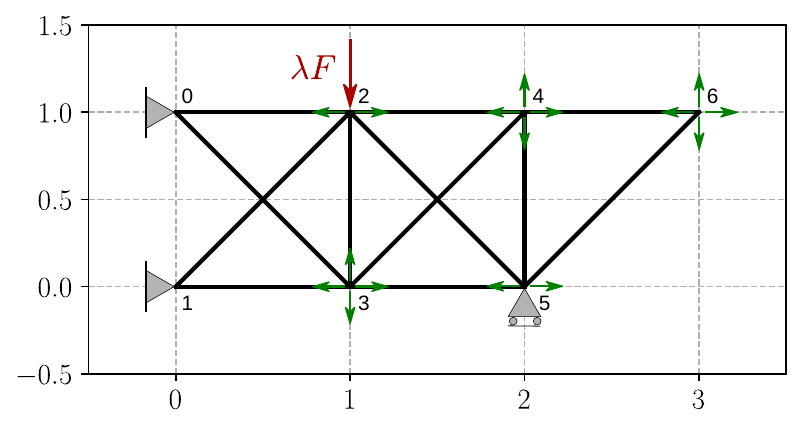}
\end{center}
\caption{A planar truss structure with a reference loading (in red) and uncertain fixed loadings (in green).}
\label{fig:truss}
\end{figure}

As an illustration, we consider the planar truss structure of Figure \ref{fig:truss} fixed at nodes 0 and 1, vertically supported at node 5 and subject to a reference vertical loading of amplitude $F$ at node 2 in nominal conditions.
 Each member $e=1,\ldots,12$ is connected with perfect hinges and therefore carries only uniform axial forces $N_e$.
 The strength condition for all members is given by $|N_e| \leq R_0$ where the axial strength $R_0$ is supposed to be the same for all members and not subject to uncertainty.
  We therefore seek the maximum force $\lambda F$ that the structure can sustain given equilibrium and strength conditions.
 Denoting by $\mathbf{N}$ the vector of unknown axial forces, by $\mathbf{f}^\text{r}$ the reference load vector and by $\mathbf{H}$ the equilibrium matrix, the nominal limit analysis problem is given by the following linear program:
\begin{equation}
\begin{array}{rll}
\displaystyle{\lambda_\text{N} = \max_{\lambda, \mathbf{N}}} & \lambda & \\
\text{s.t.} & \mathbf{H}\mathbf{N} + \lambda \mathbf{f}^\text{r} = 0 &\\
& |N_e|\leq R_0 & 1\leq e \leq 12
\end{array}
\label{eq:nominal-truss}
\end{equation}

The uncertain limit analysis problem is characterized by additional forces applied horizontally and vertically  at nodes 3--4--6 and horizontally only at nodes 2 and 5.
 Each of this 8 component is assumed to be uncertain and vary independently in the interval $[-\alpha \lambda_N F;\alpha \lambda_N F]$ for a given uncertainty amplitude $\alpha > 0$.
  As a result, the uncertain loading vector is given by:
\begin{equation}
\mathbf{f}^\text{f}(\bzeta) = \mathbf{F}^\text{f}\bzeta, \quad \text{for } \bzeta\in \Uu_\infty
\end{equation} 
where $\mathbf{F}^\text{f} \in \mathbb{R}^{9\times 8}$ is the matrix made of each uncertain loading case of individual amplitude $\alpha\lambda_N  F$ and $\bzeta\in \mathbb{R}^8$ is the vector of uncertain parameters subject to a unit box-uncertainty set.

Considering an AARC formulation with the following affine decision rule for the axial forces and the load multiplier:
\begin{subequations}
\label{eq:affine-N-lamb}
\begin{align}
\mathbf{N}(\bzeta) &= \mathbf{N}_0 + \sum_{j=1}^8 \mathbf{N}_j \zeta_j\\
\lambda(\bzeta) &= \lambda_0 + \sum_{j=1}^8 \lambda_j \zeta_j
\end{align}
\end{subequations}
the AARC formulation
 \eqref{eq:AARC-LA-loading-example} can be readily transposed to the present example as follows:
 \begin{equation}
\begin{array}{rlll}
\displaystyle{ \lambda_\text{AARC} = \max_{\bar\lambda, \mathbf{N}_i,\lambda_i, \mathbf{W}_j, \mathbf{w}_j}} &\bar\lambda&\\
\text{s.t.} &\displaystyle{\mathbf{H}\mathbf{N}_0 + \lambda \mathbf{f}^\text{r} = 0} & \\
& \mathbf{H}\mathbf{N}_j +  \lambda_j \mathbf{f}^\text{r} + \mathbf{F}^{\text{f}}_j = 0 &  1\leq j \leq 8 \\
& \displaystyle{\sum_{j=1}^m w_{ej} + \left|N_{e0} - \sum_{j=1}^8 W_{ej}\right| \leq R_0 } & 1\leq e \leq 12\\
& \max\{|W_{ej} + N_{ej}|;|W_{ej}-N_{ej}|\} \leq w_{ej} &  1\leq j \leq 8, 1\leq e \leq 12 \\
& \displaystyle{\bar\lambda + \sum_{j=1}^m |\lambda_j| \leq \lambda_0} & 
\end{array}  \label{eq:AARC-truss}
\end{equation}
Note that since the strength condition is a simple interval here, the reformulation using \thref{prop:roos-box} is exact.
 Hence, the only source of conservativeness is the use of affine decision rules \eqref{eq:affine-N-lamb}.

This problem is readily formulated and solved using the \texttt{cvxpy} package \citep{diamond2016cvxpy}.
 We first compute the nominal collapse load $\lambda_\text{N}$ by solving \eqref{eq:nominal-truss}.
 The associated collapse mechanism and distribution of optimal axial forces are represented in Figure \ref{fig:nominal-truss}.
  In particular, we see that collapse is mainly attributed to a failure in compression of the vertical member located directly below the vertical loading.
  The overhang part of the structure on the right is not involved in the collapse mechanism.

\begin{figure}
\begin{center}
\includegraphics[width=0.8\textwidth]{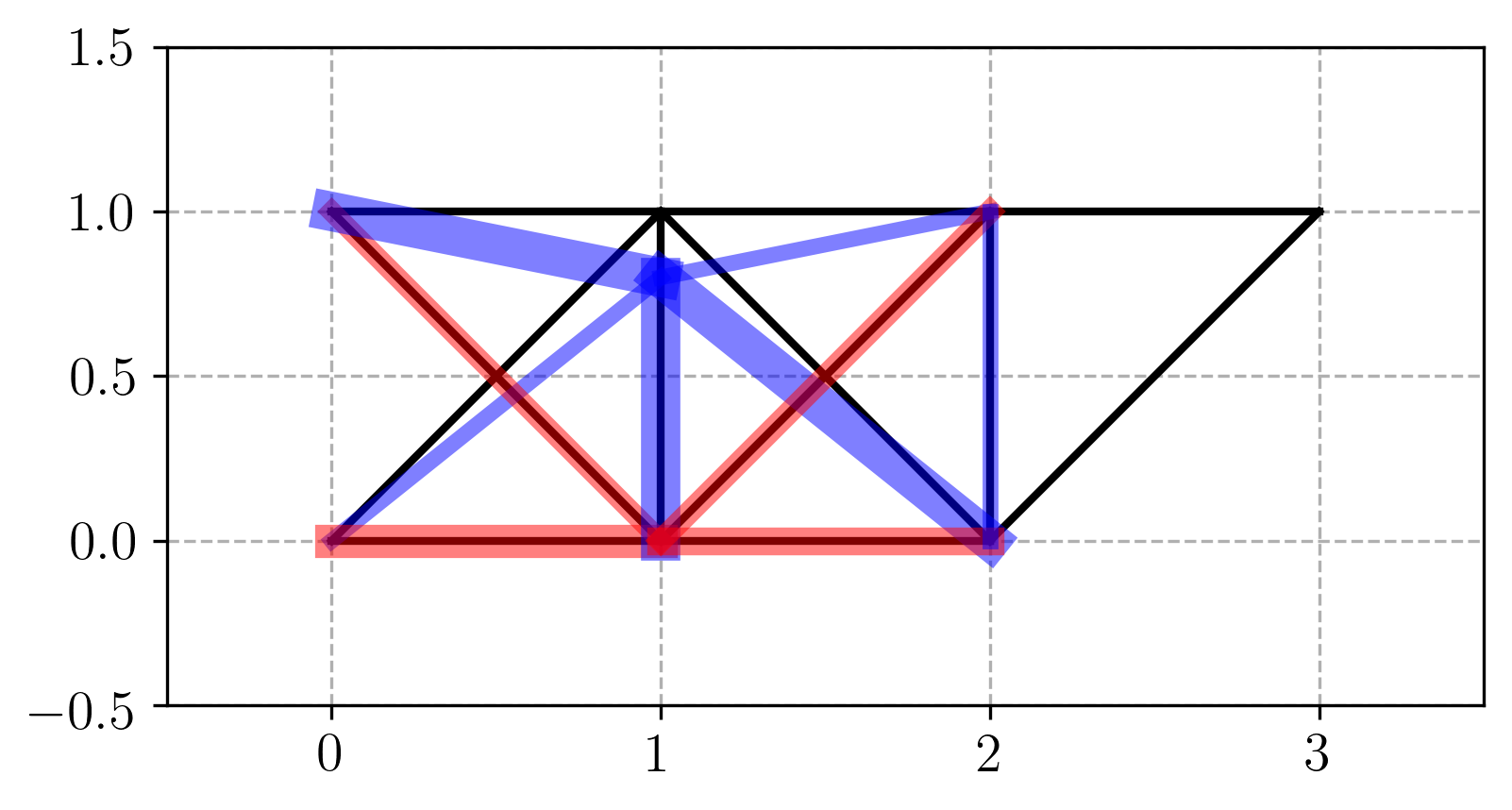}
\end{center}
\caption{Collapse mechanism and optimal axial forces (blue denotes compression $N_e < 0$ and red denotes tension $N_e>0$) for the nominal case.}
\label{fig:nominal-truss}
\end{figure}

We now solve the robust limit analysis problem \eqref{eq:AARC-truss} for various values of the uncertainty amplitude $\alpha$.
 The corresponding evolution of $\lambda_\text{AARC}$ normalized with respect to the nominal load factor $\lambda_\text{N}$ is represented Figure \ref{fig:AARC-truss}.
 It can be observed that for values of $\alpha \leq 6\%$, the robust limit load is unchanged.
 Between 6 and 12\%, the limit load decreases in a piecewise linear fashion.
 In the range 12\% to 24\%, the decrease takes place at a constant rate.
 Between 24\% and 26\%, the decrease rate is stronger and the structure eventually reaches a zero limit load.
   The AARC solution is also compared with the exact solution of the uncertain problem obtained by enumerating the $2^8$ vertices of the uncertainty set $\Uu_\infty$.
  The corresponding statistics are also reported in Figure \ref{fig:AARC-truss} and it can be seen that the AARC solution matches exactly the worst-case solution of the vertex enumeration, even in the final regime approaching a zero limit load.
  Finally, the piecewise linear evolution of the limit load with respect to $\alpha$ can be explained by a change of the corresponding collapse mechanism as a function of $\alpha$.
Indeed, when representing the worst-case collapse mechanism obtained from the vertex enumeration (see Figure \ref{fig:collapse-uncertain}), we see that each change of slope is associated with a change in the collapse mechanism and a change of the most critical vertex.
 For instance, for $\alpha = 5\%$, the mechanism is the same as in the nominal case although the axial force distribution is slightly different due to the presence of additional loadings.
 For $\alpha=10\%$, a rigid-body rotation of the right part of the structure is involved whereas for $\alpha=15\%$ or 25\% the mechanism essentially differs in the movement of the central vertical member.

\begin{figure}
\begin{center}
\includegraphics[width=0.7\textwidth]{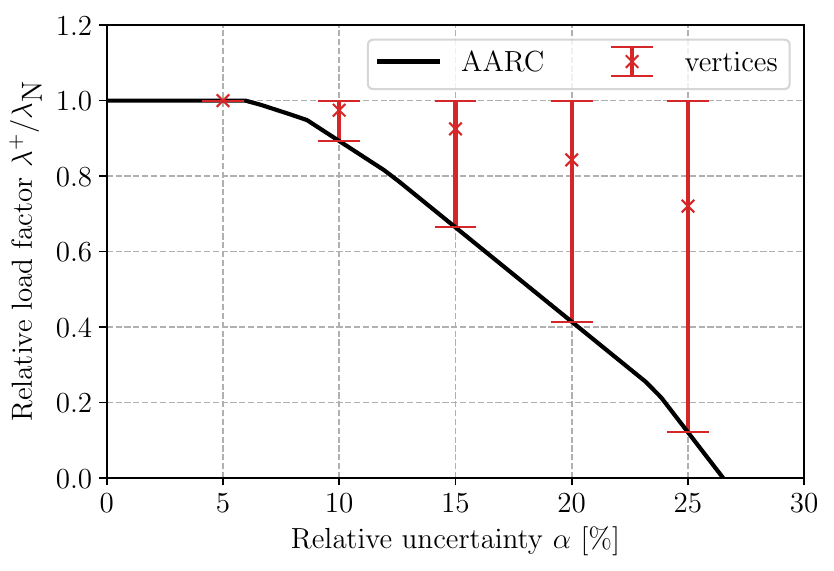}
\end{center}
\caption{Evolution of the AARC load factor as a function of the uncertainty amplitude $\alpha$.
 The "vertices" statistics denote the min, mean and maximum values of the limit load when enumerating the $2^8$ vertices of the polyhedral uncertainty set $\Uu_\infty$.}
\label{fig:AARC-truss}
\end{figure}

\begin{figure}
\begin{center}
\begin{subfigure}{0.49\textwidth}
\includegraphics[width=\textwidth]{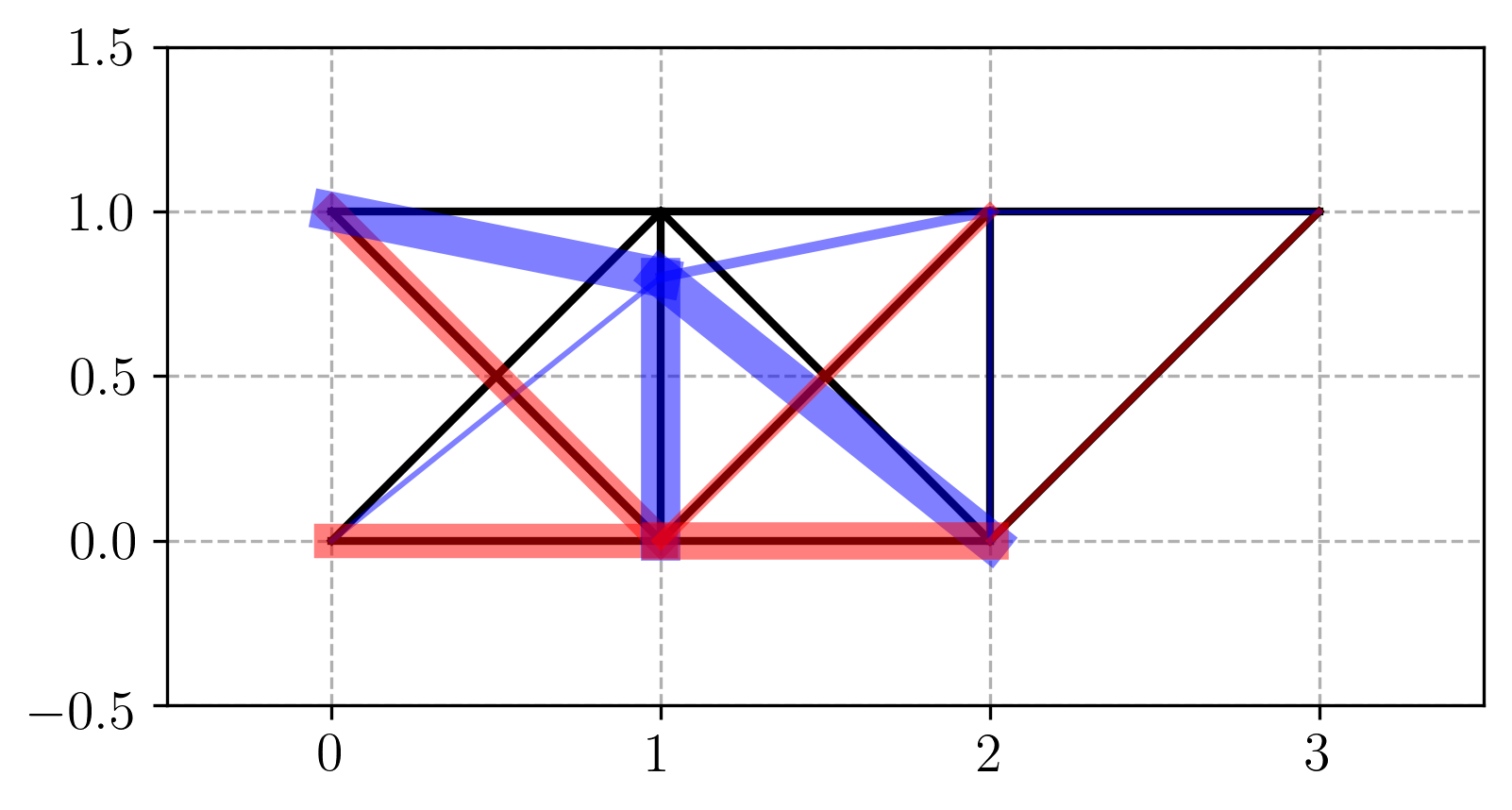}
\caption{$\alpha=5\%$}
\end{subfigure}
\hfill
\begin{subfigure}{0.49\textwidth}
\includegraphics[width=\textwidth]{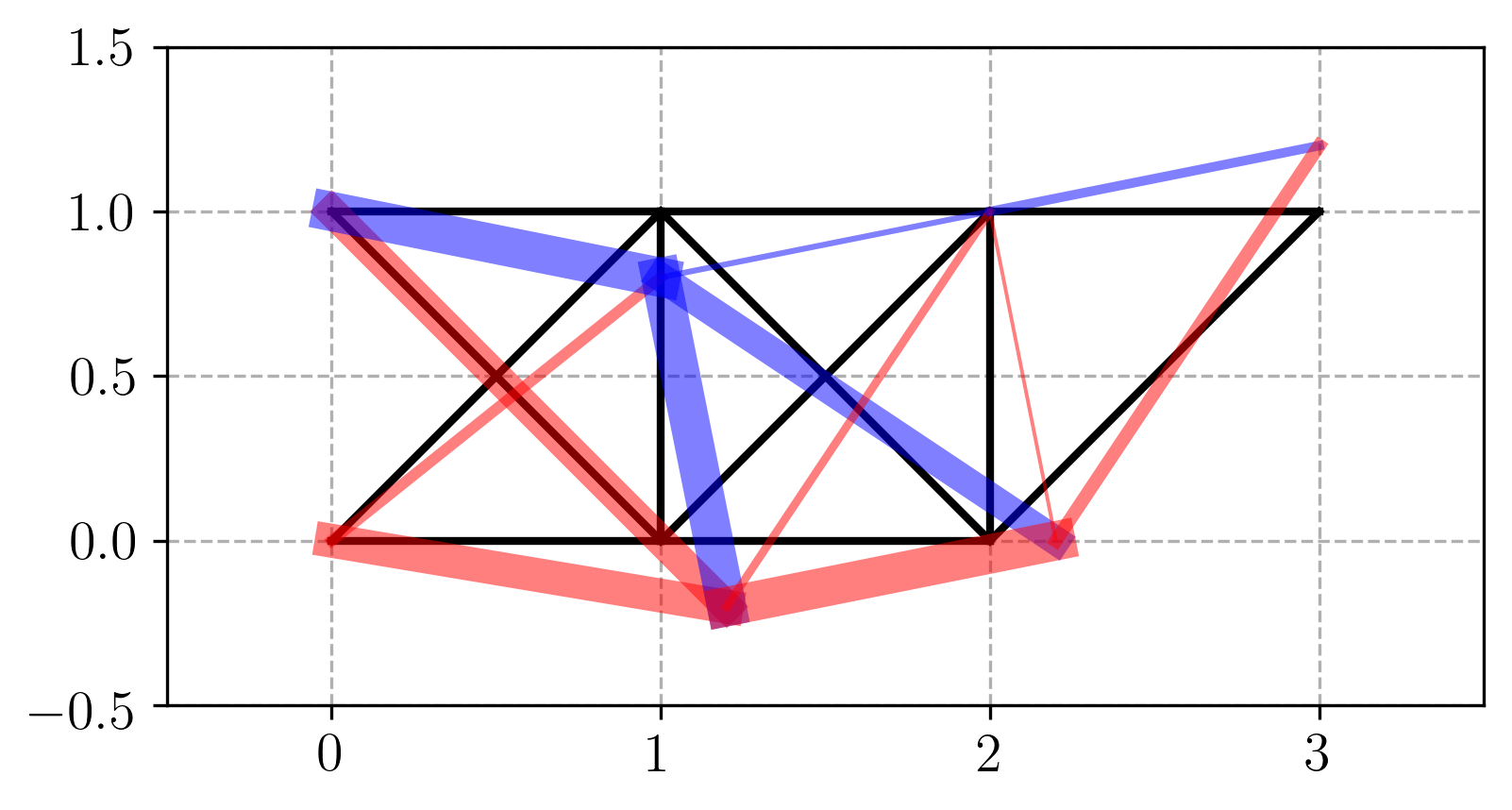}
\caption{$\alpha=10\%$}
\end{subfigure}
\begin{subfigure}{0.49\textwidth}
\includegraphics[width=\textwidth]{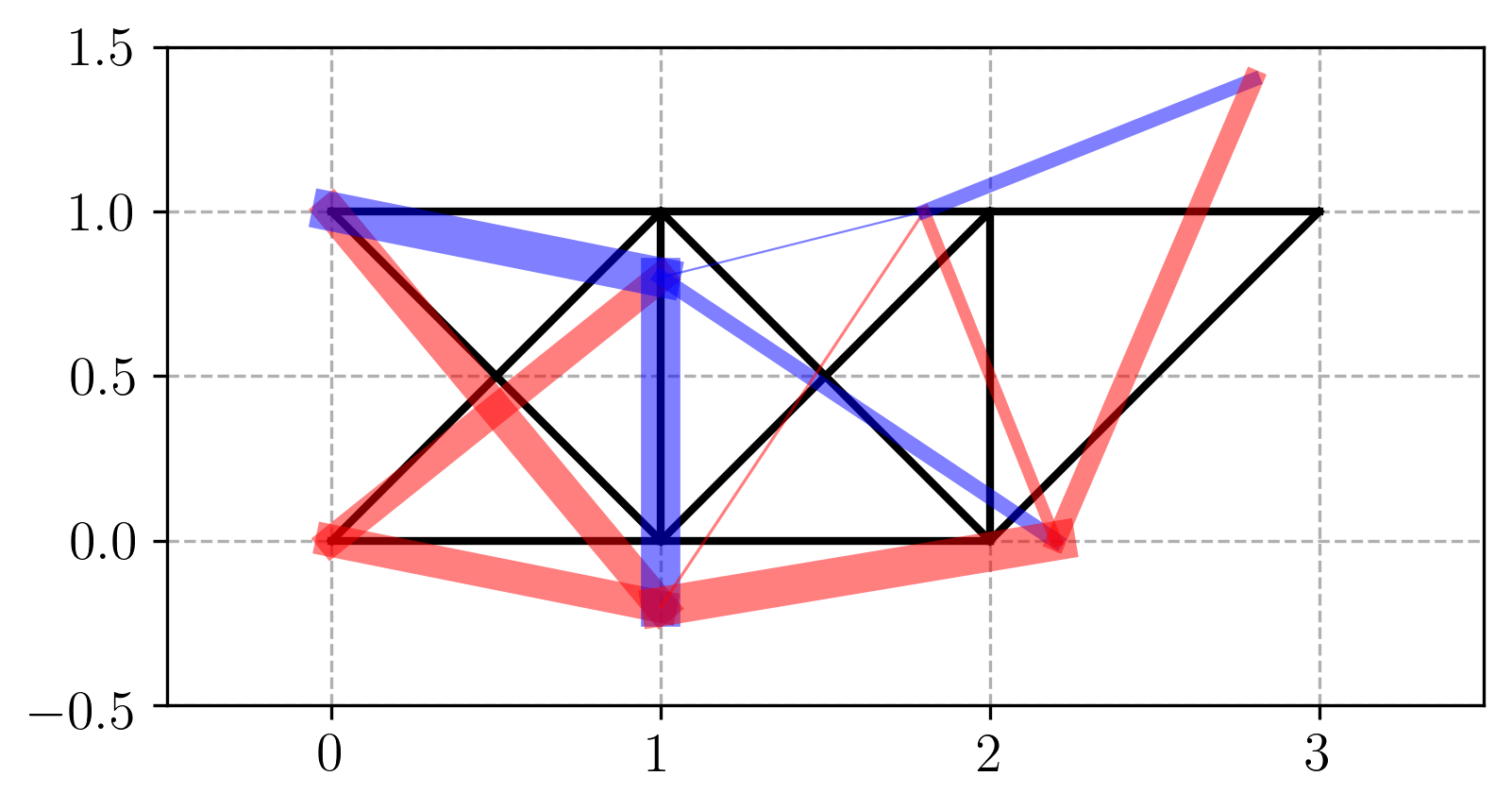}
\caption{$\alpha=15\%$}
\end{subfigure}
\hfill
\begin{subfigure}{0.49\textwidth}
\includegraphics[width=\textwidth]{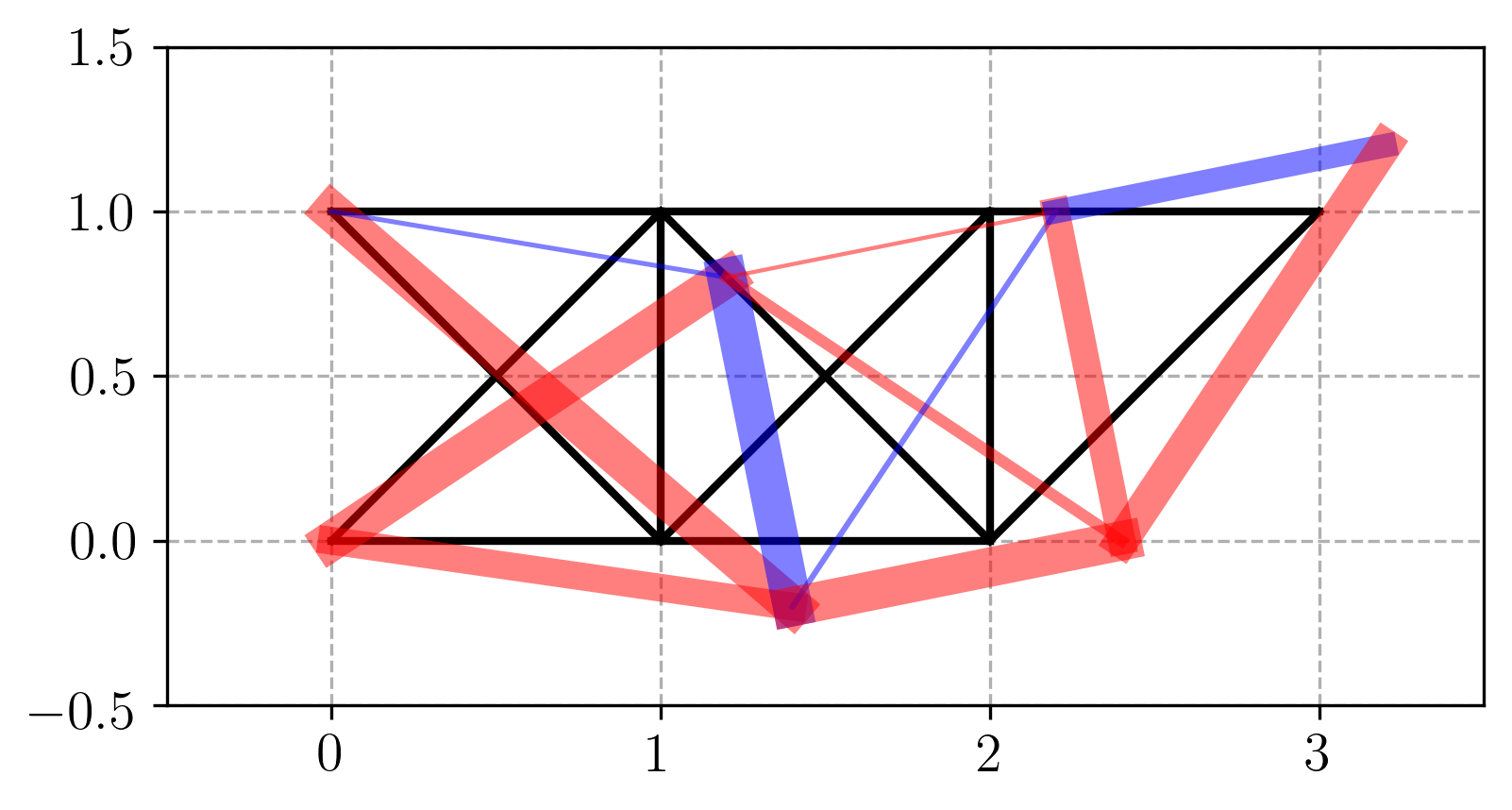}
\caption{$\alpha=25\%$}
\end{subfigure}
\end{center}
\caption{Evolution of the collapse mechanism as a function of the uncertainty parameters}
\label{fig:collapse-uncertain}
\end{figure}

\section{Conclusions}\label{sec:conclusions}

In this work, we have proposed an extension of limit analysis theory to an uncertain setting using the robust optimization (RO) framework.
 Since limit analysis problems can be formulated as convex optimization programs, we can naturally apply robust optimization concepts when considering uncertain data.
 We covered two different sources of uncertainty, namely strength and loading uncertainty.
 In the strength uncertainty case, the strength domain $G$ is assumed to be uncertain.
  The main feature of RO is to propose tractable reformulations of uncertain constraints as standard deterministic constraints.
   This step then defines a robust strength domain i.e. the smallest possible strength domain corresponding to all uncertainty realizations.
    Obtaining an explicit expression for the robust domain depends on how constraints depend on the uncertain parameters.
     We essentially explored two cases:
\begin{itemize}
\item homothetic uncertainty: the strength domain is assumed to keep its shape but depends on an uncertain scaling factor. The robust domain is obtained by minimizing this scaling factor over the uncertainty set.
\item generic form for small uncertainties: in this case, the strength domain is defined through the composition of a convex function and an uncertain linear operator which can be replaced by a first-order expansion of the uncertainty parameter.
\end{itemize}
In this second case, the uncertain strength constraint takes the form \eqref{eq:generic-robust-strength-constraint}. The fact that the constraint is a convex function of the uncertainty parameter makes the reformulation particularly challenging.
 We therefore presented both exact or conservative approximate reformulations to handle such a constraint.

Another important aspect of RO is related to the use of static or adjustable optimization variables.
 In the present case, it amounts to deciding whether we consider the stress field and load multiplier that we optimize for to be independent or dependent on the uncertain parameters.
 The former choice corresponds to static robust counterparts i.e. one aims at finding a single stress field in equilibrium and satisfying the strength conditions for any possible uncertainty realization.
  Clearly, this is a very conservative approach.
   In particular, finding such a stress field is not always possible.
    Our experience suggests that static formulations can be used only when considering strength uncertainty and in the case where this uncertainty is of small amplitude.
     Intuitively, this corresponds to the fact that the collapse stress field is only mildly perturbed by the realization of the uncertainty.
     For instance, in the case where a few elements can fail in a structure, we showed that RO formulations cannot be handled with static formulations.

One therefore needs to resort to adjustable formulations in such cases, in particular when the loading is also uncertain.
 In order to obtain tractable reformulations in the adjustable case, simple decision rules must be chosen such as affine decision rules.
 The corresponding affinely adjustable problem can then be reformulated to yield a deterministic optimization problem but involving a much larger number of optimization variables.
 In particular, robust strength constraints also take the form \eqref{eq:generic-robust-strength-constraint} in this case. 
 We have considered simple examples to illustrate the efficiency of the formulation. 
 Usually, AARC formulations seem to be very efficient although we have exhibited a specific case for which they still seem to be too conservative.\\

Further research perspectives are numerous.
 First, numerical implementation of the proposed formulations would be needed to assess their efficiency on more involved examples.
 In this respect, specific strategies should probably be investigated in order to reduce the computational cost of the corresponding large-scale optimization problems, especially when considering AARC formulations.
  Analyzing such more advanced examples would therefore shed light on the necessity, or not, of considering more complex decision rules than affine rules such as piecewise-linear or nonlinear decision rules.
   On a more fundamental level, the present work considered only a static (lower-bound) formulation of limit analysis whereas it would be interesting to also investigate the formulation of the corresponding dual problems based on a kinematic (upper-bound) limit analysis formulation.
    Finally, RO relies on the notion of uncertainty sets which consists in choosing uncertain parameters with a finite support.
    In practice, uncertain material parameters often rely on probability distributions with unbounded support (e.g. normal or log-normal distributions).
     In order to rationalize the choice of uncertainty sets, it would therefore be valuable to investigate the notion of probabilistic guarantees which can draw a parallel between uncertainty sets and probabilistic constraints in a chance-constrained programming setting.

\appendix

\section{Dual uncertainty norm}\label{app:dual-norms}
We often have to compute expressions of the form $\sup_{\bzeta\in \Uu} \{\bz\cdot\bzeta\}$ which corresponds to the \textit{support function} $\pi_{\Uu}(\bz)$ of the uncertainty set.
 Owing to the fact that the uncertainty sets of section \ref{sec:uncertainty-sets} can be represented as the unit ball with respect to some specific norm $\|\cdot\|$, we denote the corresponding support function using the dual norm $\|\cdot\|_*$: 
\begin{equation}
\pi_{\Uu}(\bz)=\|\bz\|_* = \begin{cases} 
\|\bz\|_1 & \text{for } \Uu_\infty \\
\|\bz\|_2 & \text{for } \Uu_2 \\
\|\bz\|_\infty & \text{for } \Uu_1 \\
\displaystyle{\inf_{\by}\: \{\|\by\|_1+\Gamma\|\bz-\by\|_\infty\}} & \text{for } \Uu_\text{budget}(\Gamma) 
\end{cases}
\end{equation}

One may also be interested in using a \textit{one-sided} budget uncertainty set defined as:
\begin{equation}
\Uu_\text{budget}^+(\Gamma) = \{\bzeta \in \RR^m \st 0 \leq \bzeta \leq 1 \text{ and } \|\bzeta\|_1 \leq \Gamma\}
\end{equation}
whose dual norm reads:
\begin{equation}
\|\bz\|_* = \inf_{\by}\: \{\|\operatorname{maximum}\{\by,0\}\|_1+\Gamma\|\bz-\by\|_\infty\} \label{eq:dual-norm-one-sided-budget}
\end{equation}
where the "$\operatorname{maximum}$" function is applied component-wise.

\section{Roos et al. approximation for polyhedral uncertainty}\label{app:Roos-approximation}
In the following, we adapt the formulation of \citet{roos2018approximation} to a polyhedral uncertainty set described as follows:
\begin{equation}
\Uu = \{\bzeta \in \RR^m \st \exists \bxi \in \RR^q \text{ and } \mathbf{D}_1\bzeta+\mathbf{D}_2\bxi \leq \mathbf{d}\} \label{eq:polyhedral-uncertainty}
\end{equation}
where $d\in\RR^r$, $\mathbf{D}_1\in \RR^{r\times m}$, $\mathbf{D}_2\in \RR^{r\times q}$ and $\bxi$ represents slack variables. In particular, contrary to \citet{roos2018approximation}, we do not assume the uncertainty set to belong to the positive orthant.\\

\subsection{General result}

\begin{theorem}[Roos et al. approximation for polyhedral uncertainty]\label{th:Roos-approximation}
For $\Uu$ given by \eqref{eq:polyhedral-uncertainty}, the robust constraint \eqref{eq:generic-robust-strength-constraint} can be conservatively approximated as follows  \citep{roos2018approximation}:
\begin{equation}
\exists \bv\in\mathbb{R}^r,\: \bV\in \mathbb{R}^{d\times r} \st
\begin{cases} \displaystyle{g\left(\bsig - \bV\bd\right) + \bd\T\bv \leq 1} \\
 g(\bV_i) \leq v_i \quad \forall i=1,\ldots,r\\
 \bD_1\T\bv = \bb\\
 \bD_2\T\bv = 0\\
 \bV\bD_1 + \bSig = 0\\
 \bV\bD_2 = 0
\end{cases} \label{eq:Roos-approximation}
\end{equation}
where $\bV_i$ denotes the $i$-th column of $\bV$.
\end{theorem}

\begin{proof}
We can first reformulate \eqref{eq:generic-robust-strength-constraint} by introducing the Legendre-Fenchel transform of $g$:
\begin{align}
& g(\bsig + \bSig\bzeta) \leq 1-\bb\T\bzeta, \quad \forall \bzeta \in \Uu \notag \\
\Leftrightarrow \quad & \sup_{\bzeta \in \Uu} \:\{g(\bsig + \bSig\bzeta) +\bb\T\bzeta\} \leq 1 \notag \\
\Leftrightarrow \quad & \sup_{\bzeta \in \Uu} \:\{\sup_{\bz\in \dom g^*} \:\{(\bsig + \bSig\bzeta)\T\bz - g^*(\bz)\} +\bb\T\bzeta\} \leq 1 \notag \\
\Leftrightarrow \quad & \sup_{\bz\in \dom g^*} \:\{\bsig\T\bz - g^*(\bz) + \sup_{\bzeta \in \Uu} \:\{\bz\T\bSig\bzeta + \bb\T\bzeta \}\} \leq 1 \label{eq:app:Roos-reformulation-1}
\end{align}
in which we exchanged the order of both maximizations. Note that when $g$ is the gauge function of a convex set $G$, $g^*=\delta_{G^\circ}$ is the indicator of the polar set $G^\circ$ so that $\dom g^* = G^\circ$.\\

We notice that the inner maximization over $\Uu$ is a linear program for fixed $\bz$ and we can then use strong duality of linear programming over the polyhedral set \eqref{eq:polyhedral-uncertainty} to write:
\begin{equation}
 \begin{array}{rl} \displaystyle{\sup_{\bzeta \in \Uu} \:\{\bz\T\bSig\bzeta + \bb\T\bzeta \} = \inf_{\blambda \in \mathbb{R}^r}} & \{\bd\T\blambda\} \\
 \text{s.t.} & \bD_1\T\blambda = \bSig\T\bz + \bb \\
 & \bD_2\T\blambda = 0 \\
 & \blambda \geq 0
 \end{array}
\end{equation}
So that \eqref{eq:app:Roos-reformulation-1} is equivalent to:
\begin{equation}
\forall \boldsymbol{z}\in G^{\circ}, \exists \blambda \in \RR^r \st  \begin{cases}
 \bsig\T\boldsymbol{z} - g^*(\bz) + \mathbf{d}\T\blambda \leq 1 \\
\bD_1\T\blambda = \bSig\T\bz + \bb \\
  \bD_2\T\blambda = 0 \\
  \blambda \geq 0
\end{cases}
\label{eq:app:Roos-reformulation-2}
\end{equation}
This problem is in fact an adjustable robust optimization problem for the uncertain variable $\bz\in G^\circ$ and the adjustable variable $\blambda$. 
In order to obtain a conservative and tractable approximation to \eqref{eq:app:Roos-reformulation-2}, and hence to \eqref{eq:generic-robust-strength-constraint}, we now restrict $\blambda$ to a linear decision rule such as:
 \begin{equation}
\blambda = \bv - \bV\T\bz
\end{equation}
with $\bv \in \RR^r$ and $\bV\in \RR^{d\times r}$ being now static additional variables to be optimized\footnote{Note that, without loss of generality, the minus and transposition signs are introduced to yield a more pleasing final formulation}. Classical robust optimization reformulations techniques can now be applied.\\

Let us first handle the positivity constraint $\blambda \geq 0$. Since $G$ contains the origin, we have $G^{\circ\circ} = G$, so that:
\begin{align}
& \blambda = \bv - \bV\T\bz \geq 0, \quad\forall \bz \in G^{\circ}  \\
\Leftrightarrow \quad & \sup_{\bz\in G^{\circ}} \:\{(\bV_i)\T\bz\} \leq v_i \quad \forall i=1,\ldots, r \\
\Leftrightarrow \quad & g(\bV_i) \leq v_i \quad \forall i=1,\ldots, r
\end{align}
where $\bV_i$ is $i$-th column of $\bV$.

Similarly the first constraint in \eqref{eq:app:Roos-reformulation-2} can be rewritten as:
\begin{align}
& \bsig\T\bz - g^*(\bz) + \bd\T\bv - \bd\T\bV\T\bz  \leq 1, \quad\forall \bz \in G^{\circ}  \\
\Leftrightarrow \quad & \sup_{\bz\in G^{\circ}} \:\{(\bsig-\bV\bd)\T\bz\} + \bd\T\bv \leq 1 \\
\Leftrightarrow \quad & g(\bsig-\bV\bd) + \bd\T\bv \leq 1
\end{align}
Finally, $G^\circ$ being fully dimensional, the equality constraints reduce to:
\begin{align}
\bD_1\T\bv = \bb \\
  \bD_2\T\bv = 0 \\
\bV\bD_1+\bSig = 0\\
  \bV\bD_2 = 0
\end{align}
which finally yields \eqref{eq:Roos-approximation}.
\qed
\end{proof}

\subsection{Box uncertainty}

\begin{example}[The box uncertainty case]
In the box uncertainty case $\Uu=\Uu_\infty$, \eqref{eq:polyhedral-uncertainty} corresponds to:
\begin{align}
\bD_1 &= \begin{bmatrix}
\phantom{-}\mathbf{I}_m \\ -\mathbf{I}_m
\end{bmatrix} \in \RR^{2m\times m} \notag\\
\bD_2&=0 \label{eq:box-polyhedron}\\
\bd &= \mathbf{e}_m=(1,\ldots,1)\T\in\RR^m \notag
\end{align}
We then have:
\begin{proposition}[Roos et al. approximation - box uncertainty]
If $\Uu=\Uu_\infty$, the robust constraint \eqref{eq:generic-robust-strength-constraint} can be conservatively approximated using \thref{th:Roos-approximation} as follows:
\begin{equation}
\exists \bw\in\mathbb{R}^m,\: \bW\in \mathbb{R}^{d\times m} \st
\begin{cases} \displaystyle{\sum_{j=1}^m w_j + g\left(\bsig - \sum_{j=1}^m \bW_j\right) \leq 1} \\
 g(\bW_j - \bSig_j) -b_j \leq w_j \quad \forall j=1,\ldots,m\\
 g(\bW_j + \bSig_j) + b_j \leq w_j
\end{cases}\label{eq:app:box-uncertainty-approx}
\end{equation}
\end{proposition}

\begin{proof}
The use of \thref{th:Roos-approximation} with \eqref{eq:box-polyhedron} yields:
\begin{equation}
\exists \bv^+,\bv^-\in\mathbb{R}^m,\: \bV^+,\bV^-\in \mathbb{R}^{d\times m} 
\st
\begin{cases} \displaystyle{g\left(\bsig - \sum_{j=1}^m(\bV^+_j+\bV^-_j)\right) + \sum_{j=1}^m(v_j^++v_j^-) \leq 1} \\
 g(\bV_j^+) \leq v_j^+ \quad \forall j=1,\ldots,m\\
 g(\bV_j^-) \leq v_j^- \\
 \bv^+ - \bv^- = \bb\\
 \bV^+ - \bV^- + \bSig = 0
\end{cases}\label{eq:app:box-uncertainty-approx-1}
\end{equation}

Introducing $\bW = \bV^++\bV^-$ and $\bw=\bv^++\bv^-$, we have $\bV^\pm = (\bW \mp \bSig)/2$ and $\bv^\pm = (\bw\pm\bb)/2$ so that \eqref{eq:app:box-uncertainty-approx-1} becomes:
\begin{equation}
\exists \bw\in\mathbb{R}^m,\: \bW\in \mathbb{R}^{d\times m} 
\st
\begin{cases} \displaystyle{g\left(\bsig - \sum_{j=1}^m\bW_j\right) + \sum_{j=1}^m w_j\leq 1} \\
 g((\bW_j - \bSig_j)/2) \leq (w_j+b_j)/2 \quad \forall j=1,\ldots,m\\
 g((\bW_j + \bSig_j)/2) \leq (w_j-b_j)/2 
\end{cases}\label{eq:app:box-uncertainty-approx-2}
\end{equation}
which is indeed \eqref{eq:app:box-uncertainty-approx} owing to the homogeneity of $g$.\qed
\end{proof}
\end{example}

\subsection{Cross-polytope uncertainty}
\begin{example}[The cross-polytope uncertainty case]
In the cross-polytope uncertainty case $\Uu=\Uu_1$, \eqref{eq:polyhedral-uncertainty} corresponds to:
\begin{align}
\bD_1 &= \begin{bmatrix}
\phantom{-}\mathbf{I}_m \\ -\mathbf{I}_m\\0
\end{bmatrix} \in \RR^{2m+1\times m} \notag\\
\bD_2&=\begin{bmatrix}
-\mathbf{I}_m \\ -\mathbf{I}_m\\ \mathbf{e}_m\T
\end{bmatrix} \in \RR^{2m+1\times m} \label{eq:cross-poly-polyhedron}\\
\bd &= (0, \ldots, 0, 1)\T\in\RR^{2m+1} \notag
\end{align}
We then have:
\begin{proposition}[Roos et al. approximation - cross-polytope uncertainty]
If $\Uu=\Uu_1$, the robust constraint \eqref{eq:generic-robust-strength-constraint} can be conservatively approximated using \thref{th:Roos-approximation} as follows:
\begin{equation}
\exists w\in\mathbb{R},\: \bW\in \mathbb{R}^{d} \st
\begin{cases} \displaystyle{w + g\left(\bsig - \bW\right) \leq 1} \\
 g(\bW - \bSig_j) -b_j \leq w \quad \forall j=1,\ldots,m\\
 g(\bW + \bSig_j) + b_j \leq w
\end{cases}\label{eq:app:cross-poly-uncertainty-approx}
\end{equation}
\end{proposition}

\begin{proof}
The use of \thref{th:Roos-approximation} with \eqref{eq:cross-poly-polyhedron} yields:
\begin{equation}
\begin{gathered}
\exists \bv^+,\bv^-\in\mathbb{R}^m,w\in\RR \: \bV^+,\bV^-\in \mathbb{R}^{d\times m}, \bW \in\RR^d
\st\\ 
\begin{cases} \displaystyle{g\left(\bsig -\bW\right) + w \leq 1} \\
 g(\bV_j^+) \leq v_j^+ \quad \forall j=1,\ldots,m\\
 g(\bV_j^-) \leq v_j^- \\
 g(\bW) \leq w \\
 \bv^+ - \bv^- = \bb\\
 -\bv^+ - \bv^- + w\mathbf{e}_m = \bb\\
 \bV^+ - \bV^- + \bSig = 0 \\
 -\bV^+ - \bV^- + \bW\mathbf{e}_m\T = 0
\end{cases}\label{eq:app:cross-poly-uncertainty-approx-1}
\end{gathered}
\end{equation}

From the equality constraints, we get:
\begin{equation}
\bV^\pm_j = (\bW \pm \bSig_j)/2, \quad v^\pm_j = (w \mp b_j)/2, \quad \forall j=1,\ldots,m
\end{equation}
 so that \eqref{eq:app:cross-poly-uncertainty-approx-1} becomes:
\begin{equation}
\exists \bw\in\mathbb{R},\: \bW\in \mathbb{R}^{d} 
\st
\begin{cases} \displaystyle{g\left(\bsig - \bW\right) + w\leq 1} \\
 g((\bW - \bSig_j)/2) \leq (w+b_j)/2 \quad \forall j=1,\ldots,m\\
 g((\bW + \bSig_j)/2) \leq (w-b_j)/2 \\
 g(\bW) \leq w
\end{cases}\label{eq:app:cross-poly-uncertainty-approx-2}
\end{equation}
The last constraint is in fact redundant since $g$ is homogeneous, therefore sub-additive, and we have:
\begin{equation}
g(\bW) = g\left(\dfrac{\bW+\bSig_j}{2} + \dfrac{\bW-\bSig_j}{2}\right) \leq g\left(\dfrac{\bW+\bSig_j}{2}\right) + g\left(\dfrac{\bW-\bSig_j}{2}\right) \leq w
\end{equation}
Hence, \eqref{eq:app:cross-poly-uncertainty-approx-2} is indeed \eqref{eq:app:cross-poly-uncertainty-approx}.\qed
\end{proof}

\end{example}

\subsection{Tightness compared with Bertsimas and Sim approximation}
The approximation of \thref{th:Roos-approximation} is tighter than that of \thref{th:Bertsimas-approximation} \citep{Bertsimas2020RobustCO}. For instance, taking $\bW=0$ in both \eqref{eq:app:box-uncertainty-approx} yields:
\begin{equation}
\exists \bw\in\mathbb{R}^m \st
\begin{cases} \displaystyle{\sum_{j=1}^m w_j + g(\bsig) \leq 1} \\
 g(- \bSig_j) -b_j \leq w_j \quad \forall j=1,\ldots,m\\
 g(\bSig_j) + b_j \leq w_j
\end{cases}
\end{equation}
We have $w_j \geq \max\{g(- \bSig_j) -b_j;g(\bSig_j) + b_j\}=s_j$ in \eqref{BS-s-definition}. Noticing that $g$ is positive so is $s_j$. Hence: 
\begin{equation}
\underbrace{\sum_{j=1}^m s_j}_{\|\bs\|_1=\|\bs\|_*} + g(\bsig)  \leq \sum_{j=1}^m w_j + g(\bsig)\leq 1
\end{equation}
recovering \thref{th:Bertsimas-approximation} for the box uncertainty.\\

Similarly, taking $\bW=0$ in \eqref{eq:app:cross-poly-uncertainty-approx} yields:
\begin{equation}
\exists w\in\mathbb{R} \st
\begin{cases} \displaystyle{w + g(\bsig) \leq 1} \\
 g(- \bSig_j) -b_j \leq w \quad \forall j=1,\ldots,m\\
 g( \bSig_j) + b_j \leq w
\end{cases}
\end{equation}
that is $w \geq s_j\geq 0$ $\forall j=1,\ldots,m$, hence:
\begin{equation}
\underbrace{\max_{j=1,\ldots,m}\{s_j\}}_{\|\bs\|_\infty=\|\bs\|_*} + g(\bsig)  \leq w + g(\bsig)\leq 1
\end{equation}
recovering \thref{th:Bertsimas-approximation} for the cross-polytope uncertainty.\\

\bibliographystyle{apalike}
\bibliography{robust_optimization,uncertain_LA,fenics_limit_analysis}
\end{document}